\documentclass[11pt,reqno]{amsart}
\usepackage{geometry}

\usepackage{graphicx,url,cite,amsmath,amssymb,amsfonts,color}

\usepackage{scalerel}
\usepackage{mathtools}
\usepackage[dvipsnames]{xcolor}
\usepackage{dirtytalk}

\usepackage{blindtext}
\usepackage{epigraph}

\newcommand{\rmd}{{\rm d}}

\newcommand{\bx}{{ \mathbf{x} }}

\newcommand{\by}{\mathbf{y}}
\newcommand{\tby}{\by}

\newcommand{\bFF}{\mathbf{F}}

\newcommand{\bH}{\mathbf{H}}
\newcommand{\bff}{{\mathbf{f}}}

\newcommand{\ent}{{\mathsf e\mathsf n\mathsf t}}
\newcommand{\esc}{{\mathsf e\mathsf s\mathsf c}}

\newcommand{\rnd}{{\mathsf R}}
\newcommand{\dtr}{{\mathsf D}}
\newcommand{\s}{{\mathsf s}}
\newcommand{\rr}{{\mathsf r}}
\newcommand{\syn}{{\mathsf s\mathsf y\mathsf n}}

\definecolor{mygray}{gray}{0.6}

\newcommand{\be}{\begin{equation}}
\newcommand{\ee}{\end{equation}}
\newcommand{\bea}{\begin{eqnarray}}
\newcommand{\eea}{\end{eqnarray}}

\numberwithin{equation}{section}

\newtheorem{theorem}{Theorem}
\numberwithin{theorem}{section}
\newtheorem{proposition}{Proposition}
\numberwithin{proposition}{section}
\newtheorem{corollary}{Corollary}

\numberwithin{corollary}{section}
\newtheorem{lemma}{Lemma}
\numberwithin{lemma}{section}
\theoremstyle{remark}

\theoremstyle{definition}
\newtheorem{definition}{Definition}

\begin{document}

 \title{Statistical determinism in non-Lipschitz dynamical systems}

\date{\today}

\author{Theodore D. Drivas}
\address{Department of Mathematics, Stony Brook University, Stony Brook, NY 11794}
\email{tdrivas@math.stonybrook.edu}
\author{Alexei A. Mailybaev}
\address{Instituto de Matem\'{a}tica Pura e Aplicada, Rio de Janeiro, Brazil}
\email{alexei@impa.br}
\author{Artem Raibekas}
\address{Instituto de Matem\'atica e Estat\'istica, Universidade Federal Fluminense, Niter\'oi, Brazil}
\email{artem@mat.uff.br}

\begin{abstract}
We study a class of ordinary differential equations with a non-Lipschitz point singularity, which admit non-unique solutions through this point. As a selection criterion, we introduce stochastic regularizations depending on the parameter $\nu$: the regularized dynamics is globally defined for each $\nu > 0$, and the original singular system is recovered in the limit of vanishing $\nu$.
We prove that this limit yields a \textit{unique statistical solution} independent of regularization, when the deterministic system possesses certain chaotic properties. In this case, solutions become spontaneously stochastic after passing through the singularity: they are selected randomly with an intrinsic probability distribution. 
\end{abstract}

\maketitle

\setlength{\epigraphwidth}{0.81\textwidth}
\setlength{\epigraphrule}{0pt}
\epigraph{\it ``It is proposed that certain formally deterministic fluid systems which possess many
scales of motion are observationally indistinguishable from indeterministic systems;
specifically, that two states of the system differing initially by a small ``observational
error" will evolve into two states differing as greatly as randomly chosen states of the
system within a finite time interval, which cannot be lengthened by reducing the
amplitude of the initial error."}{--- Edward N. Lorenz (1969) 
}

\section{Introduction}
\label{sec1}

Consider a nonlinear ordinary differential equation 
\begin{equation}\label{ODEa}
\frac{\rmd \bx}{\rmd t} = \bff(\bx), \quad \bx \in \mathbb{R}^d,
\end{equation}
for arbitrary dimension $d$. Local existence of 
solutions $\bx(t)$ is guaranteed if the function $\mathbf{f}:\mathbb{R}^d \mapsto \mathbb{R}^d$ is continuous, while
the Lipschitz continuity is required for its uniqueness by standard theorems. 
Breaking of the Lipschitz condition is remarkably abundant in dynamical systems modeling natural phenomena; for example, in the $n$-body problem~\cite{diacu1992singularities} or the Kirchhoff--Helmholtz system of point vortices~\cite{newton2013n}, where the forces diverge at vanishing distances. Other important examples arise in fluid dynamics, where particles are transported by shocks in compressible flows \cite{eyink2015spontaneous} or rough velocities in incompressible turbulence~\cite{frisch1995turbulence}. Many infinite-dimensional systems form singularities from smooth data in finite time; these often take the form of H\"{o}lderian cusps~\cite{eggers2015singularities}. 

The problem of fundamental importance is: how to select a ``meaningful" solution after the singularity? A natural way to answer this question is to employ a  regularization by which the system is modified (smoothed) very close to the singularity and the solution becomes well-defined at larger times. However this procedure is not robust in general; examples show it can be highly sensitive to the  regularization details~\cite{diperna1989ordinary,ciampa2019smooth2,ciampa2019smooth,drivas2018life} although unique selection is possible in some notable situations \cite{mcgehee1974triple}. In this work, we show that continuation as a stochastic process can accommodate such non-uniqueness in a natural and robust manner if the deterministic system has certain chaotic properties. 

\subsection{Model} 
\label{sec_1_1}
We consider systems (\ref{ODEa}) with the right-hand side of the form
\begin{equation}\label{ODEb}
\bff(\bx) = |\bx|^\alpha \bFF\left(\frac{\bx}{|\bx|}\right), \quad 
\bFF(\by) = \bFF_\s(\by)+F_\rr(\by)\by,
\end{equation}
where 
$
\alpha < 1$
and $\bFF: \mathbb{S}^{d-1}  \mapsto \mathbb{R}^d $ is a $C^1$-function on the unit sphere $\mathbb{S}^{d-1} = \{\by \in \mathbb{R}^d: |\by| = 1\}$ decomposed into the tangential spherical component $\bFF_\s: \mathbb{S}^{d-1}\mapsto T\mathbb{S}^{d-1} $ and the radial component $F_\rr: \mathbb{S}^{d-1} \mapsto \mathbb{R}$. 
The field $\mathbf{f}:\mathbb{R}^d \mapsto \mathbb{R}^d$ defined by \eqref{ODEb} is  continuously differentiable away from the origin. At the origin, it is only  $\alpha$--H\"{o}lder continuous for $\alpha\in (0,1)$, discontinuous for $\alpha = 0$, or divergent if $\alpha < 0$. Solutions of system \eqref{ODEb} with nonzero initial condition $\bx(0) = \bx_0$ may reach the non-Lipschitz singularity in a \textit{finite time}: 
	\begin{equation}\label{ODEb_blowup}
	\lim_{t \nearrow t_b} \bx(t) = \textbf{0}, \quad 0 < t_b < +\infty,
	\end{equation}
after which the solution is generally non-unique. 

System \eqref{ODEb} is invariant under the space-time scaling 
	\begin{equation}\label{ScSym}
	\bx \mapsto \frac{\bx}{\nu},\quad
	t \mapsto \frac{t}{\nu^{1-\alpha}}
	\end{equation}
for any constant $\nu > 0$.
This symmetry reflects, in a simplified form, the fundamental property of scale invariance in multi-scale systems~\cite{eggers2015singularities}, which feature finite-time singularities (often called \textit{blowup}). Thus, models \eqref{ODEb} represent a rather large class of singular dynamical systems that can be seen as a toy model for blowup phenomena. Following this analogy, we call (\ref{ODEb_blowup}) as blowup, interpreting $|\bx|$ as the ``scale'' of solution, and $\by = \bx/|\bx|$ as its scale-invariant (angular) part.  

For the dynamical system approach to models (\ref{ODEb}), we define the auxiliary system  for the variables $\by \in \mathbb{S}^{d-1}$ and $w \in \mathbb{R}^+$ as
	\begin{equation} 
	\label{eq_RenX2}
	\frac{\rmd \by}{\rmd \tau} =  w\bFF_\s(\by),\quad
    	\frac{\rmd w}{\rmd \tau} = w+(\alpha-1)F_\rr(\by)w^2.
	\end{equation}
Systems (\ref{ODEa})--(\ref{ODEb}) and (\ref{eq_RenX2}) are related by the transformation 
    	\begin{equation}
    	\label{eqMainTh1_introB}
	\bx = R_t(\by, w) := \left(\frac{t}{w}\right)^{\frac{1}{1-\alpha}}\by,\quad
	t = e^\tau,
    	\end{equation}
where $R_t: \mathbb{S}^{d-1}\times \mathbb{R}^+ \mapsto \mathbb{R}^d$ is the time-dependent map defined for $t > 0$. Relations (\ref{eqMainTh1_introB}) are motivated by the scaling symmetry (\ref{ScSym}), which becomes the time-translation symmetry $\tau \mapsto \tau+\tau_0$ in the autonomous system (\ref{eq_RenX2}) with the relation $\tau_0 = (\alpha-1)\log\nu$. By changing the time as $\rmd s = w \rmd \tau$, we reduce the first equation in (\ref{eq_RenX2}) to the form
	\begin{equation} 
	\label{yEqn}
	\frac{\rmd \by}{\rmd s} =  \bFF_\s(\by).
	\end{equation} 

It was shown in \cite{drivas2018life} that fixed-point and limit-cycle attractors of system (\ref{yEqn})
impose fundamental restrictions on solutions $\bx(t)$ selected by generic regularization schemes. We now extend these results for chaotic attractors leading to a conceptually different  mechanism: the long-time behavior of system (\ref{yEqn}) expressed in terms of its physical measure will define solutions selected randomly near the non-Lipschitz singularity in system (\ref{ODEa})--(\ref{ODEb}).

\subsection{Assumptions}
\label{sec_1_2}
\subsubsection{(i) On physical measures}

For each attractor $\mathcal{A}\subset \mathbb{S}^{d-1}$ of system (\ref{yEqn}), we denote its basin of attraction\footnote{A compact set $\mathcal{A}$ is an \textit{attractor} with respect to the flow $X^s$ if there exists a compact (\textit{trapping}) region $U$ satisfying $X^s(U)\subset (U)$ for all $s\geq S_0$ ($S_0$ fixed) and $\mathcal{A}=\cap_{s\geq 0} X^s(U)$\cite{Robinson}. The topological basin $\mathcal{B}(\mathcal{A})$ is the set of points that converge to $\mathcal{A}$ under the forward flow. 
} 
by $\mathcal{B}(\mathcal{A})\subset \mathbb{S}^{d-1}$ and introduce the corresponding \textit{domain of attraction} in the full phase space as the cone $\mathcal{D}(\mathcal{A}) = \{r\by: \by \in \mathcal{B}(\mathcal{A}),\ r > 0\} \subset \mathbb{R}^d$.
Denoting by $X^s:\mathbb{S}^{d-1} \mapsto \mathbb{S}^{d-1}$ the flow of system (\ref{yEqn}), we now recall the definition of a physical measure  $\mu_{\mathsf{phys}}$.  Define the basin $\mathcal{B}_{\mu_{\mathsf{phys}}}(\mathcal{A})$
with respect to the measure $\mu_{\mathsf{phys}}$ as being the set of points $\by_0\in\mathcal{B}(\mathcal{A})$ such that
\begin{equation}
    \lim_{s\rightarrow +\infty}\frac{1}{s}\int_0^s{\varphi 
    \big( X^{s_1}(\by_0)\big) \, \rmd s_1} 
    =\int \varphi(\by) \,\rmd \mu_{\mathsf{phys}} (\by),
    \label{SRBEqDef1Copy}
    \end{equation}
holds for all continuous functions $\varphi: \mathcal{B}(\mathcal{A}) \mapsto \mathbb{R}$. Then the measure $\mu_{\mathsf{phys}}$ is \textit{physical} if the basin $\mathcal{B}_{\mu_{\mathsf{phys}}}(\mathcal{A})$ has positive Lebesgue measure.  We will say that the physical measure has a 
\textit{full} basin if $\mathcal{B}_{\mu_{\mathsf{phys}}}(\mathcal{A}_+)$ coincides with  $\mathcal{B}(\mathcal{A}_+)$ Lebesgue almost everywhere. In particular, having a full basin by definition \eqref{SRBEqDef1Copy} implies the uniqueness of the physical measure with respect to the attractor $\mathcal{A}$. Let us observe that ergodic Sinai-Bowen-Ruelle (SRB) measures without zero Lyapunov exponents are also physical measures \cite{Young02}. Hyperbolic attractors\cite{BR} and the Lorenz attractor\cite{AP} give examples of systems having a unique physical (SRB) measure with a full basin.
In our formulation, we assume the existence of 
\begin{enumerate}
\item[(a)] a fixed-point attractor $\mathcal{A}_- = \{\by_-\}$ with the \textit{focusing} property $F_\rr(\by_-) < 0$; 
\vspace{3pt}
\item[(b)] a 
transitive attractor $\mathcal{A}_+$ having an ergodic physical measure $\mu_{\mathsf{phys}}$ and the \textit{defocusing} property, $F_\rr(\by) > 0$ for any $\by \in \mathcal{A}_+$,
\item[(c)] and that the physical measure $\mu_{\mathsf{phys}}$ has a full basin. 
\end{enumerate}

As shown in Section~\ref{sec_DSS}, all solutions of (\ref{ODEa})--(\ref{ODEb}) with initial condition $\bx_0 \in \mathcal{D}(\mathcal{A}_-)$ reach the non-Lipschitz singularity at the origin in finite time. On the contrary, solutions in $\mathcal{D}(\mathcal{A}_+)$ remain nonzero for arbitrarily large times. We note that the chaotic form of $\mathcal{A}_+$ is crucial for our study, while the fixed-point form of $\mathcal{A}_-$ is taken for simplify. System (\ref{yEqn}) may have other attractors in addition to $\mathcal{A}_-$ and $\mathcal{A}_+$, but they will not affect our results. 

\begin{figure}
\centering
\includegraphics[width=1\textwidth]{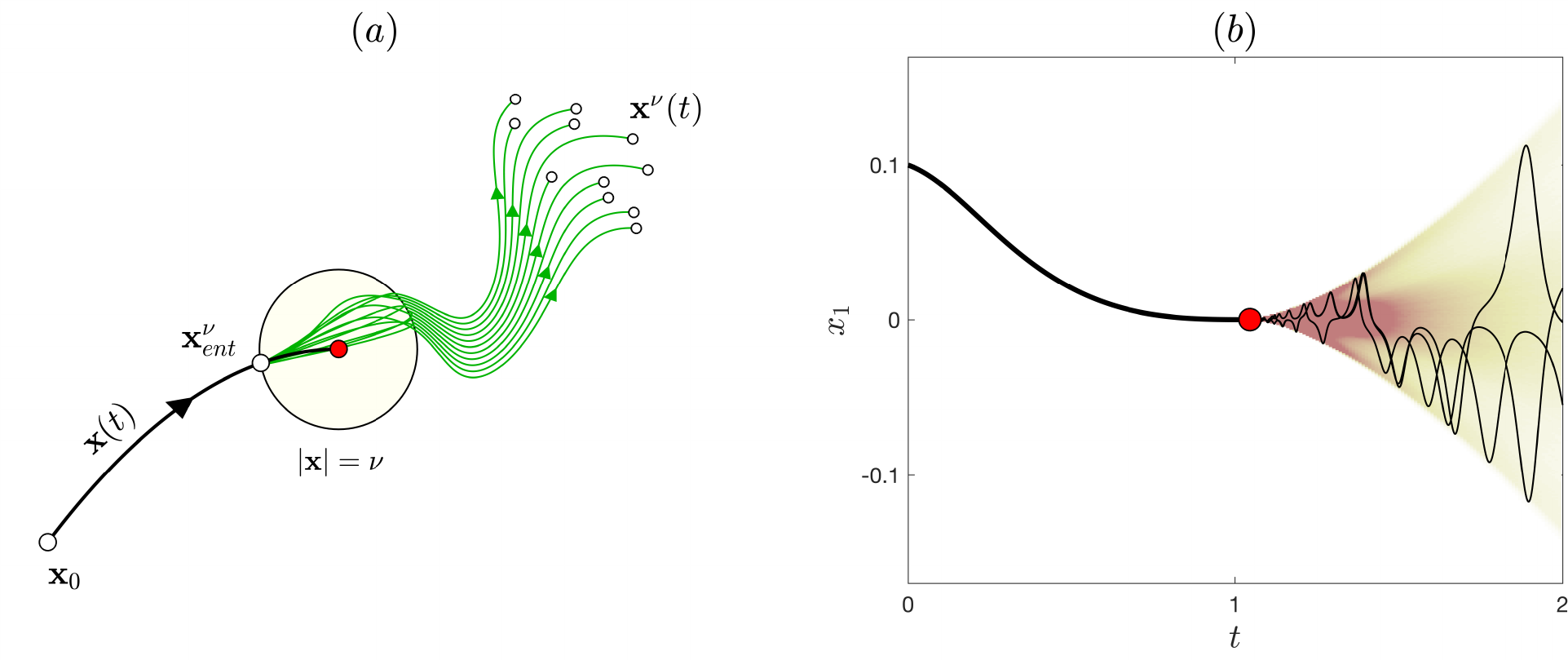}
\caption{(a) Schematic representation of the stochastic regularization procedure in the phase space $\bx \in \mathbb{R}^d$. The solution $\bx(t)$ (the black curve) starts at $\bx_0 = \bx(0)$ and reaches the singularity at $\bx(t_b) = \mathbf{0}$ in finite time. Regularized solutions (thin green curves) are given by dynamical systems smoothed in a small ball $B_\nu$ centered at the singularity. These regularizations are chosen randomly, therefore, the regularized solution is described by a time-dependent  probability measure $\bx^\nu(t) \sim \mu_t^\nu$. (b) Numerical results for the example from Section~\ref{section_Lorenz_example}. Solid lines are random realizations of component $x_1(t)$ in the regularized system with $\nu = 10^{-5}$. Color shows the probability distribution (\ref{eqMainTh1_intro}). Solutions become spontaneously stochastic passing through the non-Lipschitz singularity (red dot).}
\label{fig0}
\end{figure}

The central part of our formulation refers to a class of regularized systems, which are defined by modifying equations (\ref{ODEa})--(\ref{ODEb}) in a small ball $|\bx| < \nu$ as shown schematically in Fig.~\ref{fig0}(a). Unlike usual deterministic regularizations, we assume that our regularization contains a random uncertainty, which is characterized by an absolutely continuous probability measure. We assume certain geometrical properties of this measure related to the attractors $\mathcal{A}_-$ and $\mathcal{A}_+$. The exact definition of such regularizations is given in Section~\ref{sec_DSS} under the name of \textit{stochastic regularization of type} $\mathcal{A}_- \to \mathcal{A}_+$. The regularized system provides a unique measure-valued (stochastic) solution, $\bx^\nu(t)\sim \mu_t^\nu$, where $\mu_t^\nu$ is a probability measure depending on time $t$ and small regularization parameter $\nu$. 

We will prove that the auxiliary system (\ref{eq_RenX2}) has the property of \textit{generalized synchronization}: in the limit $s \to +\infty$, a time-independent asymptotic 
relation exists between the variables as $w = G(\by)$.
 Generalized synchronization originates from applications in nonlinear physics and communication~\cite{kocarev1996generalized}, where the variables $\by$ and $w$ are referred to as a drive and response. In our case, it yields an expression for the physical measure in system (\ref{eq_RenX2}).

\begin{proposition}[Generalized synchronization]
\label{Th1_attr}
The system (\ref{eq_RenX2}) has an attractor 
 	\begin{equation}
	\label{eqAprime}
	\mathcal{A}'_+ = 
	\{ (\by,w):\   
    w = G(\by),\ \by \in \mathcal{A}_+
    \},
    \end{equation}
where $G: \mathcal{A}_+\to \mathbb{R}^+$ is a continuous function given by
	\begin{equation}
	\label{GdefCopy}
	G(\by) = 
	\int_0^{+\infty} \exp\left[(\alpha-1)\int_0^{s_1} 
	F_\rr\left(X^{-s_2}(\by)\right)
	\rmd s_2 \right]\rmd s_1.
	\end{equation}
This attractor has the basin 
\be
\label{eq_Th1_2_basin}
 \mathcal{B}(\mathcal{A}'_+) :=\{ (\by,w) \ : \  (\by,w) \in \mathcal{B}(\mathcal{A}_+) \times\mathbb{R}^+\},
\ee
and a physical measure given by
	\begin{equation} 
	\label{eq_Th1_2}
	\rmd  \mu'_{\mathsf{phys}}(\by,w) = 
	\frac{\delta\big(w-G(\by)\big)}{c\, G(\by)} \,\rmd \mu_{\mathsf{phys}}(\by)\,\rmd w, \quad
	c = \int \frac{\rmd \mu_{\mathsf{phys}}(\by)}{G(\by)},
	\end{equation}
where $\delta$ is the Dirac delta and $c$ is the normalization factor.
\end{proposition}

\subsubsection{(ii) On convergence to equilibrium}
Consider an attractor $\mathcal{A}$ for a flow $X^s$ with a physical measure $\mu_{\mathsf{phys}}$ having a full basin. We will say that the attractor $\mathcal{A}$ has the \textit{convergence to equilibrium} property with respect to the measure $\mu_{\mathsf{phys}}$ when
    	\begin{equation}
    	\lim_{s \rightarrow +\infty}\int {\varphi \circ X^s \, \rmd \mu(\by)}
    	=\int \varphi \,\rmd \mu_{\mathsf{phys}} (\by)
    	\label{SRBEqDef1CE_s}
    	\end{equation}
for all absolutely continuous probability measures $\mu$ supported in the basin $\mathcal{B}(\mathcal{A})$ and all bounded continuous functions $\varphi:  \mathcal{B}(\mathcal{A}) \to \mathbb{R}$.
Notice that condition (\ref{SRBEqDef1CE_s}) refers to statistical averages at large but fixed times, unlike the condition (\ref{SRBEqDef1Copy}) on the physical measure, which is applied to temporal averages along specific solutions. The convergence to equilibrium property is guaranteed, e.g., for hyperbolic flows~\cite{BR}[Theorem 5.3]. Now let $Y^\tau: \mathcal{B}(\mathcal{A}'_+) \mapsto \mathcal{B}(\mathcal{A}'_+)$ be the flow of system (\ref{eq_RenX2}) in the basin $\mathcal{B}(\mathcal{A}'_+)$ given by (\ref{eq_Th1_2_basin}). Our final assumption is that
\begin{enumerate}
\item[(d)] the physical measure $\mu'_{\mathsf{phys}}$ of the attractor $\mathcal{A}'_+$ given by (\ref{eq_Th1_2}) in Proposition \ref{Th1_attr} has the property of convergence to equilibrium.
\end{enumerate}
It is then natural to ask what are the conditions on the vector field on the sphere (\ref{yEqn}), having the attractor $\mathcal{A}_+$ and the physical measure $\mu_{\mathsf{phys}}$, so that the above assumption is satisfied. Certain sufficient conditions are established in Section~\ref{sec_results}, which now summarize. Let us suppose that 
the attractor $\mathcal{A}_+$ of system \eqref{yEqn} with the physical measure $\mu_{\mathsf{phys}}$ satisfies the convergence to equilibrium property.
Consider a closed subset $\mathcal{V} \subset \mathbb{S}^{d-1}$, such that its complement $\mathbb{S}^{d-1}\setminus \mathcal{V}$ contains $\mathcal{A}_-$ and is contained in the interior of $\mathcal{B}(\mathcal{A}_-)$. We will further assume that 
\begin{enumerate}
\item[(i)] there exists a  constant $F_0 > 0$ such that $F_\rr(\by) = F_0$ for any $\by \in \mathcal{V}$, 
\item[(ii)] $\|\nabla \bFF_\s\| < (1-\alpha)F_0$ for the operator norm of the Jacobian matrix and any $\by \in \mathcal{V}$.
\end{enumerate}
In particular, under these hypothesis the results of Section~\ref{sec_results}, Proposition \ref{Prop2} state that the physical measure of the attractor $\mathcal{A}'_+$ in system \eqref{eq_RenX2} also has convergence to equilibrium.
As is discussed in Section~\ref{sec_results}, this permits to conclude the existence of examples satisfying the assumptions (a)--(d).

\subsection{Formulation of the main result}
\label{sec_1_2b}

\begin{theorem} [Spontaneous stochasticity] 
\label{th1_conv}
Given an arbitrary initial condition $\bx_0 \in \mathcal{D}(\mathcal{A}_-)$, there exists a finite time $t_b > 0$ such that the solution $\bx(t)$ of system (\ref{ODEa})--(\ref{ODEb}) is nonzero in the interval $t \in [0,t_b)$ and reaches the singularity $\bx(t_b) = \mathbf{0}$. There exists a measure-valued solution $\mu_t$ for $t > t_b$ with the following properties:
\begin{itemize}
    \item[(i)] For any $t > t_b$, the measure 
        	\begin{equation}
    	\label{eqTh1_1}
	\mu_t = \lim_{\nu \searrow 0}\mu_t^\nu
    	\end{equation}
is a weak limit of the regularization procedure;
    \item[(ii)] The measures $\mu_{t}$ are supported in $\mathcal{D}(\mathcal{A}_+)$ and satisfy the dynamic relation 
    	\begin{equation}
    	\label{eqTh1_2}
	\mu_{t_2} = \left(\mathrm{\Phi}^{t_2-t_1}\right)_*\mu_{t_1}\quad
	\textrm{for any} \quad t_2 > t_1 > t_b,
    	\end{equation}
where the asterisk denotes the pushforward and $\mathrm{\Phi}^t$ is the flow of system \eqref{ODEa}--\eqref{ODEb}. Also, 
    	\begin{equation}
    	\label{eqTh1_3}
	\lim_{t\searrow t_b}\mu_t(\bx) = \delta^d(\bx)
    	\end{equation}
converges to the Dirac measure of the deterministic singular state $\bx(t_b) = \mathbf{0}$.
\end{itemize}
The solution $\mu_t$ is independent of regularization and given explicitly as 
    	\begin{equation}
    	\label{eqMainTh1_intro}
	\mu_t = \left(R_{t-t_b}\right)_* \mu'_{\mathsf{phys}},
    	\end{equation}
with the measure $\mu'_{\mathsf{phys}}$ from (\ref{eq_Th1_2}) and the map $R_t$ introduced for $t > 0$ in (\ref{eqMainTh1_introB}).
\end{theorem}

There are two fundamental implications of Theorem~\ref{th1_conv}. First, it shows that the limit $\nu \searrow 0$ of a stochastically regularized solution exists. This limit yields a stochastic solution for the original singular system (\ref{ODEa})--(\ref{ODEb}): even though the random perturbation formally vanishes in the limit $\nu \searrow 0$, a random path is selected at $t > t_b$; see Fig.~\ref{fig0}(b) demonstrating numerical results from the example presented in Section~\ref{section_Lorenz_example}. Such behavior substantiates the fundamental role of infinitesimal randomness in the regularization procedure of non-Lipschitz systems, and this phenomenon is termed \textit{spontaneous stochasticity}.

The second implication is that the spontaneously stochastic solution is insensitive to a specific choice of the stochastic regularization, within the class of regularizations under consideration. The reason, which is also an underlying idea of the proof, is the following: we show that an interval between $t_b$ and any finite time $t > t_b$ in system (\ref{ODEa})--(\ref{ODEb}) can be represented by an infinitely large time interval for system (\ref{eq_RenX2}) as $\nu \searrow 0$. As a result, a random uncertainty introduced by the infinitesimal regularization  develops into the unique physical measure. This relates the spontaneous stochasticity in our system with chaos or, more specifically, with the convergence to equilibrium property for a chaotic attractor. 

Notice that the case when $\mathcal{A}_+$ is a fixed point is also a special case of our theory. In this case a unique deterministic solution is selected at times $t > t_b$ independently of regularization, as shown previously in \cite{drivas2018life}.

\subsection{Spontaneous stochasticity in models of fluid dynamics}
\label{sec_disc}

Our work provides a class of relatively simple mathematical models, where one can access sophisticated aspects of spontaneous stochasticity: its detailed mechanism, dependence on regularization and robustness. We regard these models as toy descriptions of the spontaneous stochasticity phenomenon in hydrodynamic turbulence, where singularities and small noise are known to play important role~\cite{leith1972predictability,ruelle1979microscopic,eyink1996turbulence}. Below we provide a short survey guiding an interested reader through more sophisticated models from this field.

First, we would like to mention the prediction of Lorenz \cite{lorenz1969predictability} (see the epigraph above), in which he envisioned that the role of uncertainty in multi-scale fluid models may be fundamentally different from usual chaos. Spontaneous stochasticity can be encountered in the Kraichnan model for a passive scalar advected by a H\"{o}lder continuous (non-Lipschitz) Gaussian  velocity~\cite{bernard1998slow}. Here, the statistical solution emerges in a suitable zero-noise limit and describes non-unique particle trajectories~\cite{eijnden2000generalized,le2002integration,le2004flows,drivas2017lagrangian1,kupiainen2003nondeterministic}; see also related studies for one-dimensional vector fields with H\"{o}lder-type singularities~\cite{bafico1982small,attanasio2009zero,trevisan2013zero,flandoli2013topics}. Similar behavior is encountered for particle trajectories in Burgers solutions at points of shock singularities \cite{eyink2015spontaneous} and quantum systems with singular potentials \cite{eyink2015quantum}. 
The uniqueness of statistical solutions has been tested numerically for shell models of turbulence~\cite{mailybaev2016spontaneous,mailybaev2016spontaneously,mailybaev2017toward,biferale2018rayleigh}, and in the dynamics of singular vortex layers~\cite{simon2020}.
We note that the prior work on shell models together with recent numerical studies~\cite{campolina2018chaotic,dallaston2018discrete} demonstrate chaotic behavior near non-Lipschitz singularities, when solutions are represented in renormalized variables and time. This is similar to our model, in which the spontaneous stochasticity is related to chaos in a smooth renormalized dynamical system (\ref{yEqn}). 

\subsection{Structure of the paper.} Section~\ref{sec1} contains the introduction and formulation of the main result. Section~\ref{sec_DSS} describes basic properties of solutions and defines the stochastic regularization. Section~\ref{section_Lorenz_example} contains a numerical example inspired by the Lorenz system. Section~\ref{sec_results} provides further developments with the focus on the construction of theoretical examples having robust spontaneous stochasticity. All proofs are collected in Section~\ref{sec_proofs}. 

\section{Definition of regularized solutions}
\label{sec_DSS}
 
First let us show how non-vanishing solutions $\bx(t)$ of the singular system \eqref{ODEa}--\eqref{ODEb} are described in terms of solutions $\by(s)$ for system \eqref{yEqn}.
\begin{proposition}
\label{propSI}
Let $\tby(s)$ solve \eqref{yEqn} for $s \ge 0$ with initial condition $\by(0) = \by_0$ and let 
	\begin{equation} \label{newtime}
	t_b = \lim_{s \to +\infty}t(s),\qquad
	t(s) = \int_{0}^s {r}^{1-\alpha}(s_1)\, \rmd s_1,\qquad 
	r(s) = r_0\exp\int_0^{s} F_\rr(\tby(s_1))\rmd s_1,
	\end{equation}
for any given $r_0 > 0$.  Then, the solution $\bx(t)$ of \eqref{ODEa}--\eqref{ODEb} for $t\in [0,t_b)$ with initial data $\bx(0) = r_0\by_0$ is given by $\bx(t) = r(s(t))\by(s(t))$, where $s:[0,t_b) \mapsto \mathbb{R}^+$ is the inverse of the function $t(s)$ defined in \eqref{newtime}. If $t_b$ is finite, the solution has the blowup property (\ref{ODEb_blowup}).
\end{proposition}

This statement can be checked by the direct substitution into \eqref{ODEa}--\eqref{ODEb}; see \cite{drivas2018life} for details. The next statement, also proved in \cite{drivas2018life}, refers to the focusing and defocusing attractors of system \eqref{yEqn}, $\mathcal{A}_-$ and $\mathcal{A}_+$, which were introduced in Section~\ref{sec_1_2}. 

\begin{proposition}
\label{propSol}
Solutions $\bx(t)$ of system \eqref{ODEa}--\eqref{ODEb} with initial conditions $\bx_0 \in \mathcal{D}(\mathcal{A}_-)$ have the blowup property (\ref{ODEb_blowup}) with $\mathrm{dist}\left(\bx/|\bx|, \mathcal{A}_-\right) \to 0$ as $t \nearrow t_b$. Solutions with $\bx_0 \in \mathcal{D}(\mathcal{A}_+)$ remain in $\mathcal{D}(\mathcal{A}_+)$ at all times $t > 0$  with the monotonously increasing $|\bx|$ and $\mathrm{dist}\left(\bx/|\bx|, \mathcal{A}_+\right) \to 0$ as $t \to +\infty$.
\end{proposition}

Let us illustrate these properties with the two-dimensional example~\cite{drivas2018life} for $\alpha = 1/3$ and
	\begin{equation}
	\bFF(\by) = \left(\begin{array}{c}
    y_1^2+y_1y_2+y_1y_2^2\\
	y_1y_2+y_2^2-y_1^2y_2
	\end{array}\right), \quad
	\bFF_\s(\by) = (y_1y_2^2,-y_1^2y_2), \quad 
	F_\rr(\by) = y_1+y_2,
	\label{eqEx1}
	\end{equation}
where $\by = (y_1,y_2) \in \mathbb{S}^1$ belongs to the unit circle on the plane. Dynamics on a circle of the scale-invariant system \eqref{yEqn} is shown in Fig.~\ref{fig1}(a) and the corresponding solutions of the singular system \eqref{ODEa}--\eqref{ODEb} in Fig.~\ref{fig1}(b). The focusing fixed-point attractor at $(-1,0)$ features blowup solutions, which occupy the corresponding domain $\mathcal{D}(\mathcal{A}_-) = \{(x_1,x_2) \in \mathbb{R}^2: x_1 < 0\}$. There is also a defocusing fixed-point attractor at $(1,0)$. Its domain $\mathcal{D}(\mathcal{A}_+)  = \{(x_1,x_2) \in \mathbb{R}^2: x_1 > 0\}$ contains solutions growing unboundedly at long times. This example demonstrates the strong non-uniqueness for all solutions starting in the left half-plane: they can be extended beyond the singularity in uncountably many ways. 

\begin{figure}[t]
\centering
\includegraphics[width=0.7\textwidth]{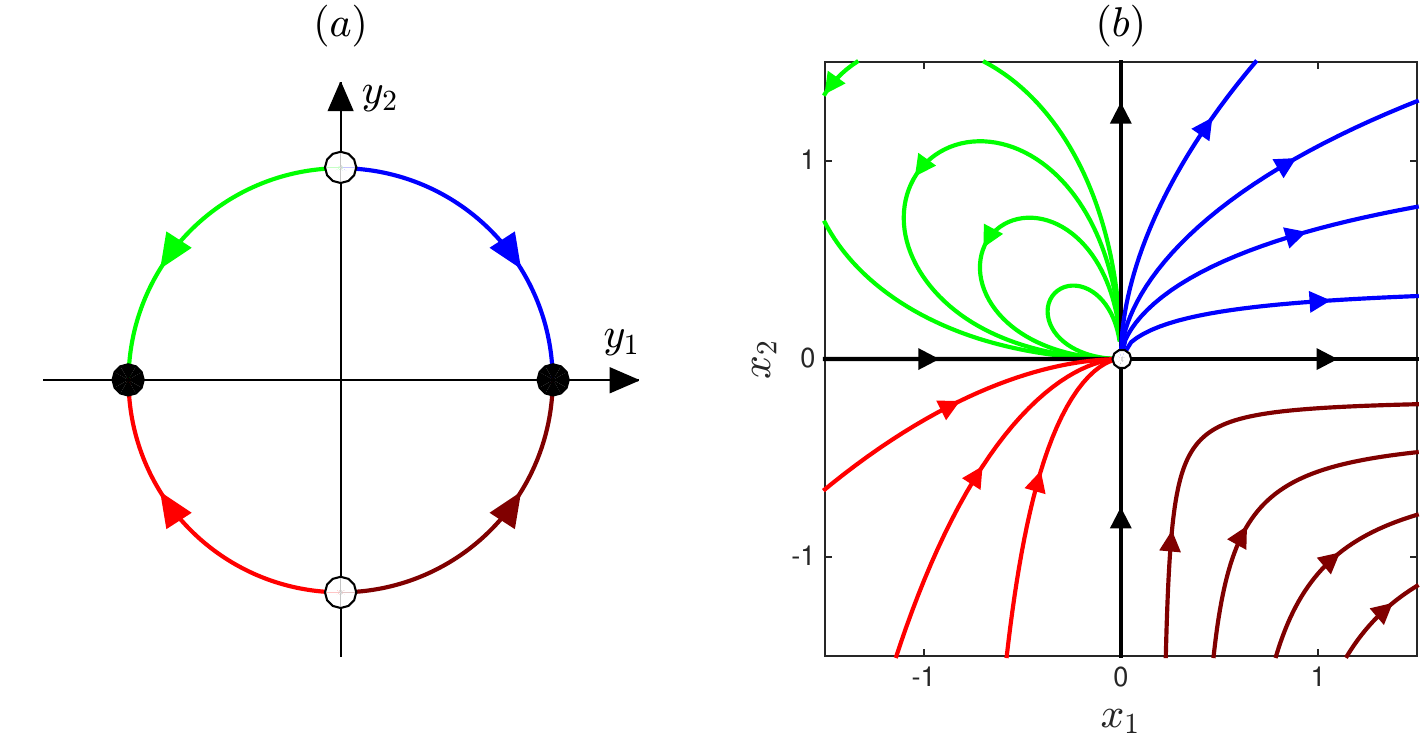}
\hspace{5mm}
\caption{(a) Dynamics of the scale-invariant system \eqref{yEqn} on the unit circle for example (\ref{eqEx1}). There are two attractors (black dots): focusing on the left and defocusing on the right. (b) Solutions of system \eqref{ODEa}--\eqref{ODEb}. Colored curves  correspond to solutions of the same color in the left panel.}
\label{fig1}
\end{figure}

\subsection{Regularized sytem} 
\label{ssec:nuRegA}

Let us consider a class of $\nu$-\textit{regularized systems} 
	\begin{align}\label{eqn:epsRegCopy}
	\displaystyle
 	\frac{\rmd \bx}{\rmd t} = \bff^\nu(\bx) ,\quad   \bff^\nu(\bx):= 
 	\begin{cases} 
	\displaystyle
	|\bx|^\alpha\bFF(\bx/|\bx|), & \bx \notin B_\nu,\\[3pt]
	\nu^\alpha \bH(\bx/\nu), & \bx \in B_\nu,
 	\end{cases}
	\end{align}
where $\nu > 0$ is the regularization parameter and $B_\nu = \{\bx \in \mathbb{R}^d: |\bx| \le \nu\}$ is the ball of radius $\nu$; recall that $\alpha < 1$. Here $\bH: B_1 \mapsto \mathbb{R}^d$ is a $C^1$-function in the unit ball, which is chosen so that the resulting regularized field $ \bff^\nu$ is $C^1(\mathbb{R}^d)$. This property does not depend on $\nu$, because the function $\bH$ in system (\ref{eqn:epsRegCopy}) is scaled to match the self-similar form of the original vector field. Note that the described choice of regularization leaves large freedom due to its dependence on the function $\bH$. 
The regularized field $\bff^\nu$ recovers the original singular system \eqref{ODEb} by taking the limit $\nu \searrow 0$. Motivated by the conceptual similarity with the viscous regularization acting at small scales in fluid dynamics~\cite{frisch1995turbulence}, we call $\nu$ the \textit{viscous parameter} and the limit $\nu \searrow 0$ the \textit{inviscid limit}. 
 
The scaling symmetry (\ref{ScSym}) extends to system (\ref{eqn:epsRegCopy}) as
	\begin{equation}\label{ScSymR}
	\bx \mapsto \frac{\bx}{\nu},\quad
	t \mapsto \frac{t}{\nu^{1-\alpha}}, \quad \nu \mapsto 1.
	\end{equation}
Let us denote the flow of the regularized system (\ref{eqn:epsRegCopy}) by $\mathrm{\Phi}^t_\nu: \mathbb{R}^d \mapsto \mathbb{R}^d$; it is uniquely defined for $t \ge 0$, $\nu > 0$ and $\alpha < 1$. Symmetry (\ref{ScSymR})  yields the relation between the regularized flows for arbitrary $\nu > 0$ and $\nu  = 1$ as
	\begin{equation}\label{ScSymRR}
	\mathrm{\Phi}^t_{\nu}(\bx) = \nu\, \mathrm{\Phi}^{t/\nu^{1-\alpha}}_1\left(\frac{\bx}{\nu}\right).
	\end{equation}
Using this map for a deterministically or randomly chosen function $\bH$, we now introduce the two types of regularizations: deterministic and stochastic.

\subsection{Deterministic regularization of type $\mathcal{A}_- \to \mathcal{A}_+$}

\begin{figure}
\centering
\includegraphics[width=0.9\textwidth]{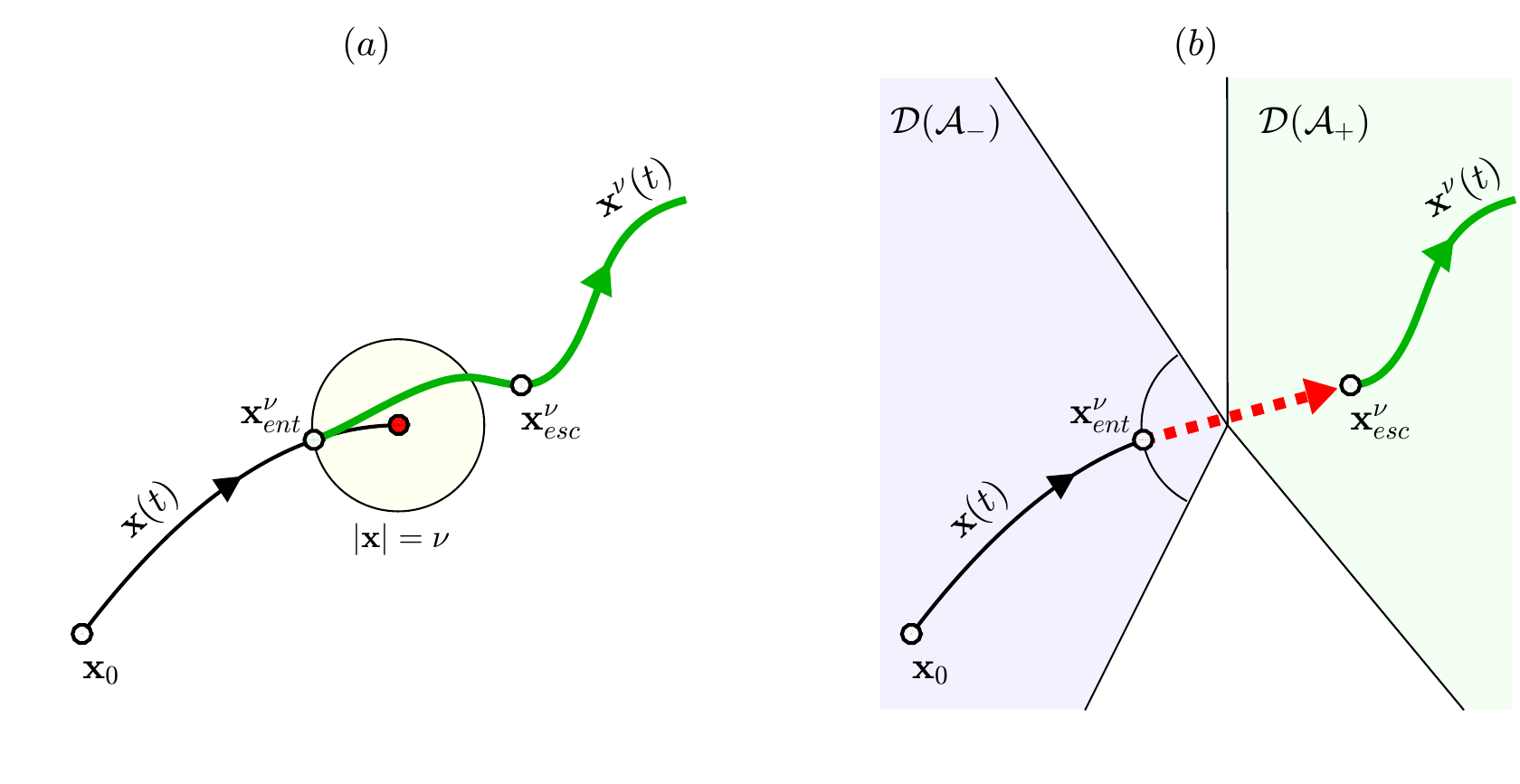}
\caption{Schematic representation of the regularization procedure in the phase space $\bx \in \mathbb{R}^d$. (a) The blowup solution $\bx(t)$ (the black curve) starts at $\bx_0 = \bx(0)$ and reaches the singularity at $\bx(t_b) = \mathbf{0}$ in finite time. The regularized solution $\bx^\nu(t)$ (thick green curve) is given by the dynamical system modified in a small ball $B_\nu$ centered at the singularity. The solutions $\bx(t)$ and $\bx^\nu(t)$ coincide until and differ after the point $\bx_\ent^\nu$.
(b) This regularization procedure is formalized by considering the two segments: the original solution $\bx(t)$ until the \textit{entry} point $\bx_\ent^\nu$, and the regularized solution $\bx^\nu(t)$ after the \textit{escape} point $\bx_\esc^\nu$. The two points $\bx_\ent^\nu$ and $\bx_\esc^\nu$ are related via the regularization map $\mathrm{\Psi}_{\dtr}$ represented by the bold dashed arrow. For the regularization of type $\mathcal{A}_- \to \mathcal{A}_+$, the first segment belongs to the domain $\mathcal{D}(\mathcal{A}_-)$, while the second segment belongs to $\mathcal{D}(\mathcal{A}_+)$.}
\label{fig2}
\end{figure}

Consider initial condition  $\bx_0 \in \mathcal{D}(\mathcal{A}_-)$ in the domain of the focusing attractor. The corresponding solution $\bx(t)$ of system \eqref{ODEb} reaches the origin in finite time $t_b$; see Proposition~\ref{propSol}. Let us consider the solution $\bx^\nu(t)$ of regularized system (\ref{eqn:epsRegCopy}) with the same initial condition for a given (small) viscous parameter $\nu > 0$. This solution exists and is unique globally in time. The two solutions $\bx(t)$ and $\bx^\nu(t)$ coincide up until the first time, when the solution enters the  ball $B_\nu$; see Fig.~\ref{fig2}(a). We denote this \textit{entry time} by $t_\ent^\nu$, which has the properties
	\begin{equation}
	\label{eq_invlim}
	t_\ent^{\nu} < t_b,\quad
	\lim_{\nu \searrow 0} t_\ent^{\nu} = t_b.
	\end{equation}
We assume that the regularization is such that the solutions spend only a finite time in the ball $B_\nu$. Then, we can select an \textit{escape time} $t_\esc^\nu > t_\ent^\nu$ such that $\bx^\nu(t) \notin B_\nu$ for $t \ge t_\esc^\nu$. The corresponding \textit{entry} and \textit{escape points} are denoted by 
	\begin{equation}
	\label{eq_reg_sol_ent1}
	\bx_\ent^\nu = \bx(t_\ent^\nu),\quad
	\bx_\esc^\nu = \bx(t_\esc^\nu),
	\end{equation}
and have $|\bx_\ent^\nu| = \nu$ and $|\bx_\esc^\nu| > \nu$; see Fig.~\ref{fig2}(a).
Using (\ref{ScSymRR}), points (\ref{eq_reg_sol_ent1}) are related by
	\begin{equation}
	\label{eq_reg_sol_EscP}
	\frac{\bx^\nu_\esc}{\nu} = \mathrm{\Psi}_{\dtr}
	\left(\frac{\bx^\nu_\ent}{\nu}\right), \quad 
	t^\nu_\esc = t^\nu_\ent+\nu^{1-\alpha} T,
	\end{equation}
for the map $\mathrm{\Psi}_{\dtr} = \mathrm{\Phi}_1^{T}$ and $T > 0$ defined by the second equality. We assume that escape points belong to the domain of a focusing attractor, $\bx_\esc^\nu \in \mathcal{D}(\mathcal{A}_+)$, ensuring that $\bx^\nu(t) \in \mathcal{D}(\mathcal{A}_+)$ at all larger times; see Proposition~\ref{propSol}.

The limit in (\ref{eq_invlim}) implies that $|\bx^\nu_\ent/\nu-\by_-| \to 0$ as $\nu \searrow 0$ for the fixed-point attractor $\mathcal{A}_- = \{\by_-\}$; see Proposition~\ref{propSol}. Hence, for our purposes, it is sufficient to know the map in (\ref{eq_reg_sol_EscP}) for a small neighborhood $\mathcal{U}_- \subset \mathbb{S}^{d-1}$ of $\by_-$. Let us choose a sufficiently small neighborhood and $T > 0$ such that $|\mathrm{\Phi}_1^{T}(\by)| > 1$ for all $\by \in \mathcal{U}_-$. Then, we can define the escape points by relations (\ref{eq_reg_sol_EscP}), where both $T$ and the map $\mathrm{\Psi}_{\dtr} = \mathrm{\Phi}_1^{T}$ do not depend on $\nu$ or initial point $\bx_0$. This yields 

\begin{definition}[Deterministic regularization]
\label{def_attr3Copy}
A \emph{regularization of type $\mathcal{A}_- \to \mathcal{A}_+$} is given by a continuous map 
	\begin{align}
	\label{ODErescaleMapCopy}
	\mathrm{\Psi}_{\dtr}:\ \mathcal{U}_- \mapsto \mathcal{D}(\mathcal{A}_+),
	\end{align}
which is defined in some neighborhood $\mathcal{U}_- \subset \mathbb{S}^{d-1}$ of $\mathcal{A}_-$, and a constant $T > 0$. For $\bx_0 \in \mathcal{D}(\mathcal{A}_-)$, escape points are given by relations (\ref{eq_reg_sol_EscP}) for sufficiently small $\nu > 0$.
\end{definition}

Having the escape point and time, one defines the regularized solution simply as
	\begin{equation}
	\label{eqRR_5}
	\bx^\nu(t) = \mathrm{\Phi}^{t-t_\esc^\nu}\left(\bx^\nu_\esc\right), 
	\quad
	t \ge t^\nu_\esc,
	\end{equation}
where $\mathrm{\Phi}^t$ is the flow of the original singular system \eqref{ODEa}--\eqref{ODEb}. In the limit $\nu \searrow 0$, we will not be interested in the solution inside the vanishing interval $t \in (t^\nu_\ent,\, t^\nu_\esc)$. Therefore, for our purposes, the regularization process is conveniently represented by the single map $\mathrm{\Psi}_{\dtr}$ in the generalized Definition~\ref{def_attr3Copy}. We do not need to explicitly specify a regularizing field $\bH$, which generated this map; see Fig.~\ref{fig2}(b).

\subsection{Stochastic regularization}
\label{subsec_stocreg}

\begin{figure}
\centering
\includegraphics[width=0.5\textwidth]{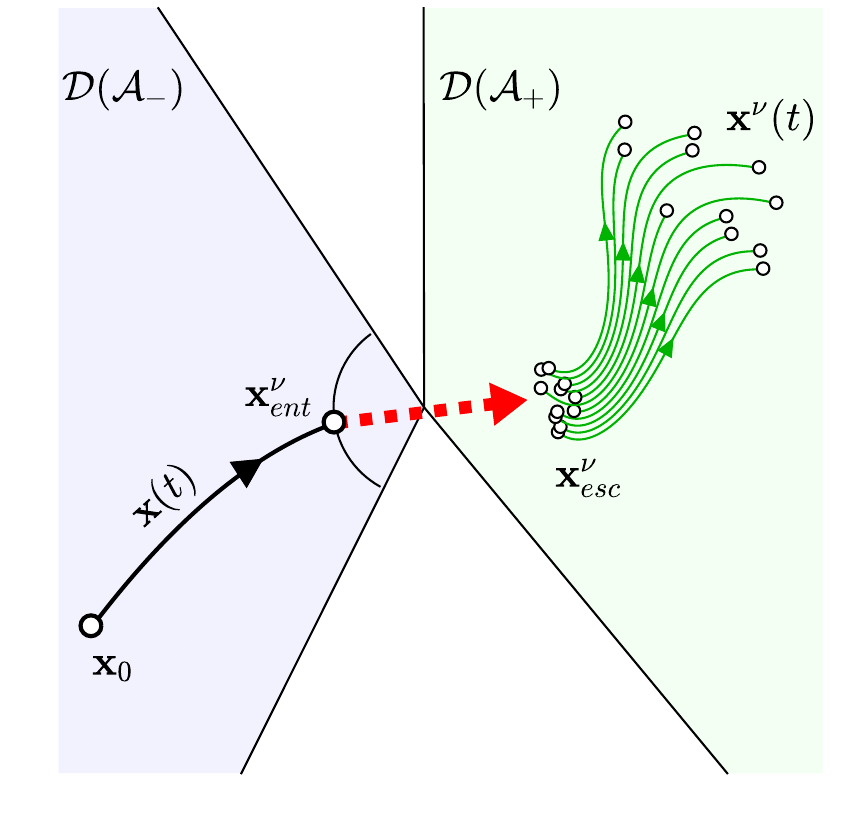}
\caption{Schematic representation of the stochastic regularization procedure in the phase space $\bx \in \mathbb{R}^d$. The solution contains the two segments: the original deterministic solution $\bx(t)$ until the entry point $\bx_\ent^\nu$, and the regularized solution $\bx^\nu(t)$ emanating from the random escape point  $\bx_\esc^\nu$. The probability distribution of $\bx_\esc^\nu$ is related to the entry point $\bx_\ent^\nu$ via the regularization map $\mathrm{\Psi}_\rnd$. For the regularization of type $\mathcal{A}_- \to \mathcal{A}_+$, the first segment belongs to the domain $\mathcal{D}(\mathcal{A}_-)$ and the second to $\mathcal{D}(\mathcal{A}_+)$.}
\label{fig3}
\end{figure}

It is known that, in general, solutions $\bx^\nu(t)$ with deterministic regularization do not converge in the inviscid limit $\nu \searrow 0$~\cite{drivas2018life}. The limits may exist along some subsequences $\nu_n \searrow 0$ but need not be unique. We now introduce a different type of regularization by assuming that escape points are known up to some random uncertainty; see Fig.~\ref{fig3}. For this purpose, one may consider a family of regularized systems (\ref{eqn:epsRegCopy}) with the field $\bH$ depending of a vector of parameters, and impose a probability distribution on values of these parameters; see Section~\ref{section_Lorenz_example} for an explicit example. Assuming that the resulting probability distributions are described by absolutely continuous measures, the regularization map (\ref{ODErescaleMapCopy}) is substituted by the function
	\begin{align}
    	\label{eqSR1}
    	\mathrm{\Psi}_\rnd:\ \mathcal{U}_-  \mapsto L^1(\mathcal{D}(\mathcal{A}_+)).
    \end{align}
This map defines a (positive and unit $L^1$ norm) probability density function
	\begin{align}
    	\label{eqSR1b}
	f^\nu_\esc = \mathrm{\Psi}_\rnd\left(\frac{\bx^\nu_\ent}{\nu}\right) 
	\in L^1(\mathcal{D}(\mathcal{A}_+))
    \end{align}
supported in the domain $\mathcal{D}(\mathcal{A}_+)$. Using the scaling (\ref{eq_reg_sol_EscP}), we introduce the probability distribution $\mu^\nu_\esc$ for a random escape point as 
	\begin{align}
    	\label{eqSR1_add2}
	\rmd \mu^\nu_\esc(\bx) = f^\nu_\esc\left(\frac{\bx}{\nu}\right) \frac{\rmd \bx}{\nu},\qquad
	t^\nu_\esc = t^\nu_\ent+\nu^{1-\alpha} T.
    	\end{align}
Summarizing and adding the continuity condition, we propose the following definition.

\begin{definition}[Stochastic regularization]
\label{def_stoch_reg}
A \emph{stochastic regularization of type $\mathcal{A}_- \to \mathcal{A}_+$} is given by a continuous map (\ref{eqSR1}), which is defined in some neighborhood $\mathcal{U}_- \subset \mathbb{S}^{d-1}$ of $\mathcal{A}_-$, and a constant $T > 0$. For $\bx_0 \in \mathcal{D}(\mathcal{A}_-)$, the entry point $\bx^\nu_\ent/\nu \in \mathcal{U}_-$ for sufficiently small $\nu > 0$ and random escape points are given the probability distribution $\mu^\nu_\esc$ defined in (\ref{eqSR1b}) and (\ref{eqSR1_add2}).
\end{definition}
 
We define the measure-valued \textit{stochastically regularized solution} $\bx^\nu(t) \sim \mu^\nu_t$ as 
	\begin{equation}
	\label{eqRR_5st}
	\mu^\nu_t = 
	\left(\mathrm{\Phi}^{t-t_\esc^\nu}\right)_* 
	\mu^\nu_\esc, 
	\quad
	t \ge  t^\nu_\esc,
	\end{equation}
where the asterisk denotes the push-forward of measure $\mu^\nu_\esc$ by the flow $\mathrm{\Phi}^t$ of original singular system \eqref{ODEa}--\eqref{ODEb}. Similarly to (\ref{eqRR_5}), the solution is now defined at all times except for a short interval $(t_\ent^\nu,t_\esc^\nu)$ vanishing as $\nu\searrow 0$. 

Definition~\ref{def_stoch_reg} completes the formulation of our main result in Theorem~\ref{th1_conv}. This theorem states that, when the randomness of regularization is removed in the limit $\nu \searrow 0$, the limiting solution exists. This limit is independent of regularization and intrinsically random (spontaneously stochastic): different solutions are selected randomly at times $t >  t_b$ with the uniquely defined probability distribution.

\section{Spontaneous stochasticity with Lorenz attractor: numerical example}
\label{section_Lorenz_example}

In this section, we design an explicit example of singular system \eqref{ODEb} with the exponent chosen as $\alpha = 1/3$, and observe numerically the spontaneously stochastic behavior. We consider this example for the dimension $d = 4$, which is the lowest dimension allowing chaotic dynamics \eqref{yEqn} on the unit sphere, $\by = (y_0,y_1,y_2,y_3) \in S^3$. The radial field is chosen as $F_\rr(\by) = -y_0$. The tangent vector field $\bFF_\s$ is defined as the interpolation between two specific fields $\bFF_-$ and $\bFF_+$ in the form
	\begin{equation}
	\label{eq_example_1}
	\bFF_\s(\by) = S_1(\xi) \bFF_-(\by) + \left(1-S_1(\xi)\right)\bFF_+(\by),\quad
	\xi = 2y_0-1/2,
	\end{equation}
where $S_1$ the is the smoothstep (the cubic Hermite) interpolation function 
	\begin{equation}
	\label{eq_example_1b}
	S_1(\xi) = \left\{
	\begin{array}{ll}
	0,& \xi \le 0;\\[3pt]
	3\xi^2-2\xi^3,\ & 0 \le \xi \le 1;\\[3pt]
	1, & 1 \le \xi.
	\end{array}
	\right.
	\end{equation}
The function $\bFF_\s$ coincides with $\bFF_-$ in the upper region $y_0 \ge 0.75$ and with $\bFF_+$ in the lower region $y_0 \le 0.25$; see Fig.~\ref{fig_ex1}. We take $\bFF_-(\by) = \mathrm{P}_s(0,-y_1,-2y_2,-3y_3)$, where $\mathrm{P}_s$ is the operator projecting on a tangent space of the unit sphere. This field has the fixed-point attractor $\mathcal{A}_- = \{\by_-\}$ at the ``North Pole'' $\by_- = (1,0,0,0)$, which is the node with eigenvalues $-1$, $-2$ and $-3$. This attractor is focusing because $F_\rr(\by_-) = -1$. 
We choose the field $\bFF_+(\by)$ such that its flow is diffeomorphic to the flow of the Lorenz system 
	\begin{equation}
	\label{eq_example_4}
	\dot x = 10(y-x), \quad \dot y = x\left(28-z\right)-y, \quad \dot z = xy-8z/3
	\end{equation}
by the scaled stereographic projection 
	\begin{equation}
	\label{eq_example_4b}
	x = \frac{40y_1}{1-y_0},\quad
	y = \frac{40y_2}{1-y_0},\quad
	z = 38+\frac{40y_3}{1-y_0}.
	\end{equation}
This projection is designed such that the lower hemisphere, $y_0 < 0$, contains the Lorenz attractor $\mathcal{A}_+$; see Fig.~\ref{fig_ex1}. It is defocusing, because $F_\rr(\by) = -y_0 > 0$. 

\begin{figure}
\centering
\includegraphics[width=0.65\textwidth]{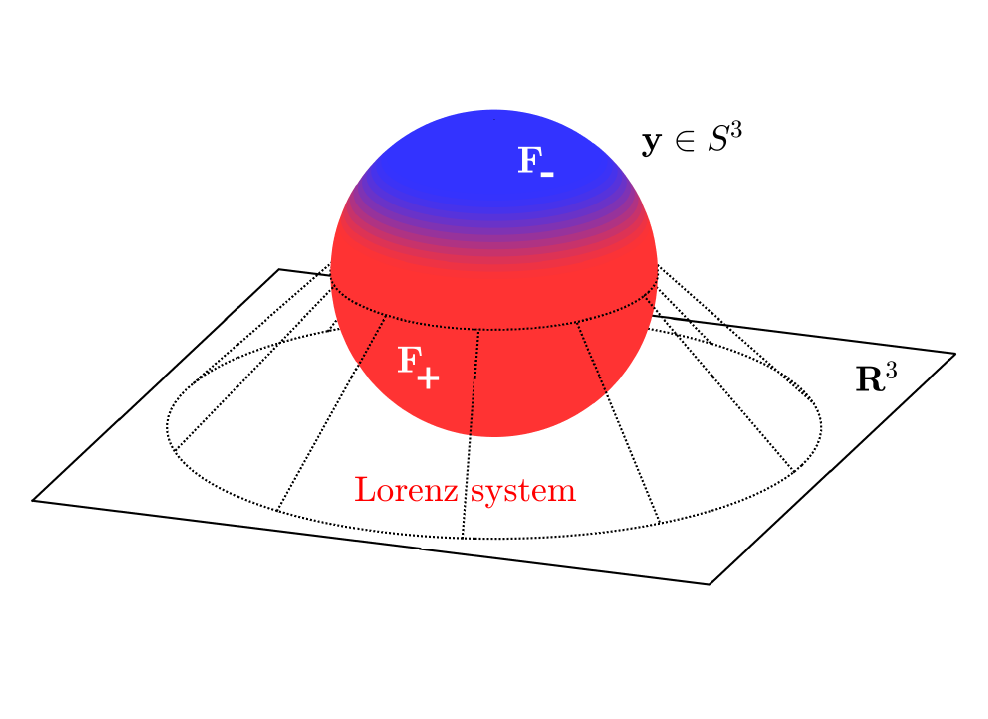}
\vspace{-10mm}
\caption{Schematic structure of the spherical field $\bFF_\s(\by)$ in our example. It is composed of the field $\bFF_-$ in the blue region, which has the fixed-point attractor at the ``North Pole'', and the field $\bFF_+$ in the red region, which is diffeomorphic to the Lorenz system. The fields are patched together using a smooth interpolation.}
\label{fig_ex1}
\end{figure}

In system (\ref{eqn:epsRegCopy}), we use the regularized field
	\begin{equation}
	\label{eq_example_5}
	\bH(\bx) = S_1(\eta)\bH_0
	+\left(1-S_1(\eta)\right)\bff(\bx),\quad
	\eta = 2|\bx|-1/2,
	\end{equation}
which interpolates smoothly between the original singular field $\bff(\bx)$ for $|\bx| \ge 3/4$ and the constant field $\bH_0$ for $|\bx| \le 1/4$. The latter is chosen as $\bH_0 = (X_0,X_1,X_2,X_3-1)$, where $X_i$ are time-independent random numbers uniformly distributed in the interval $[-1/2,1/2]$. We confirmed numerically that such a field induces the stochastic regularization of type  $\mathcal{A}_- \to \mathcal{A}_+$ according to Definition~\ref{def_stoch_reg}. 

It is expected but not known whether the flow of the Lorenz system has the property of convergence to equilibrium, as required in Theorem~\ref{th1_conv}. Therefore, with the present example we verify numerically that the concept of spontaneous stochasticity extends to such systems. We perform high-accuracy numerical simulations of systems \eqref{ODEa}--\eqref{ODEb} and (\ref{eqn:epsRegCopy})  with  the Runge--Kutta fourth-order method. The initial condition is chosen as $\bx_0 = (0.4,0.1,0.2,0.3)$. 
The solution $\bx(t)$ of singular system \eqref{ODEa}--\eqref{ODEb} reaches the origin at $t_b \approx 1.046$ (blowup). Figure~\ref{fig_ex2} shows regularized solutions for three random realizations of the regularized system with the tiny $\nu = 10^{-5}$. One can see that these solutions are distinct at post-blowup times.

\begin{figure}
\centering
\includegraphics[width=1.0\textwidth]{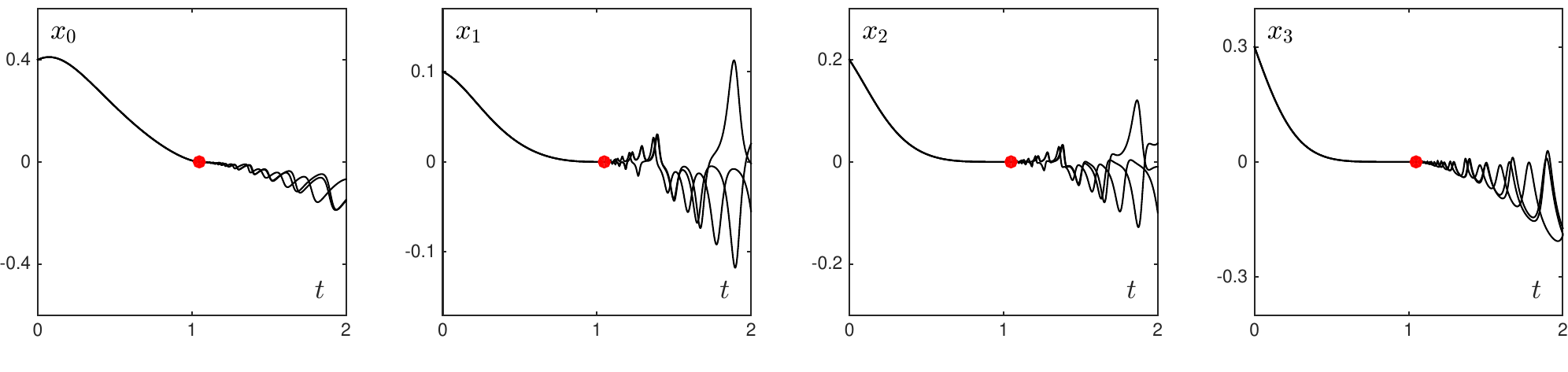}
\caption{Components $(x_0,x_1,x_2,x_3)$ of regularized solutions $\bx^\nu(t)$ for $\nu = 10^{-5}$ for three random choices of vector $\bH_0$ in regularized field (\ref{eq_example_5}). These solutions are different after the blowup time $t_b \approx 1.046$; the blowup point is indicated by the red dot.}
\label{fig_ex2}
\end{figure}

\begin{figure}
\centering
\includegraphics[width=0.65\textwidth]{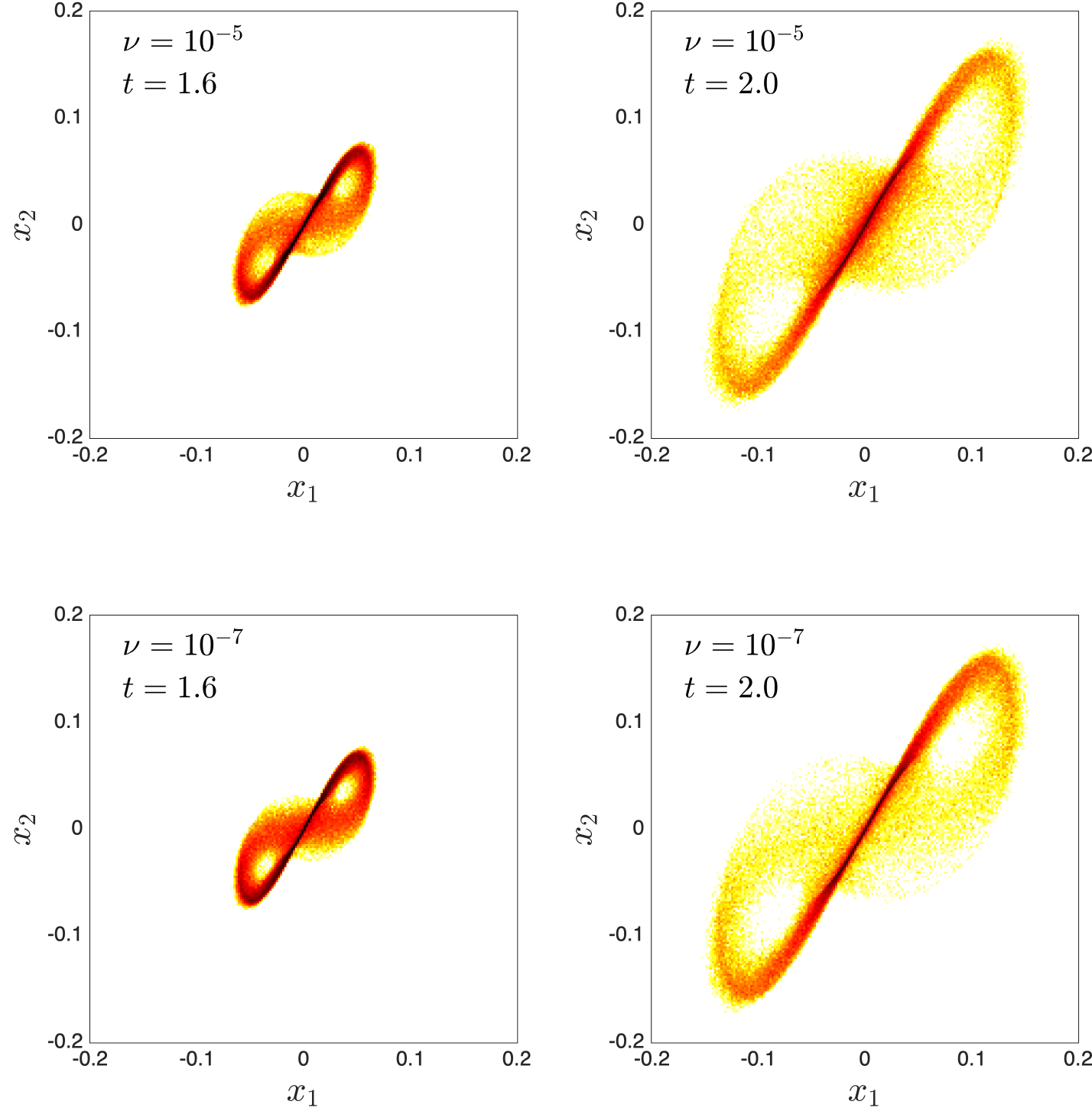}
\caption{Probability density computed numerically at times $t = 1.6$ (left) and $t = 2.0$ (right) using the statistical ensemble of $10^5$ regularized solutions. The darker color indicates the higher density. The first row corresponds to $\nu = 10^{-5}$ and the second row to $\nu = 10^{-7}$, confirming the spontaneous stochasticity in the inviscid limit.}
\label{fig_ex3}
\end{figure}

In order to observe the spontaneous stochasticity, we compute numerically the probability density for the regularized solution projected on the plane $(x_1,x_2)$ at two post-blowup times: $t = 1.6$ and $2.0$. This is done by considering an ensemble of $10^5$ random realizations of the regularized field, and the results are shown in Fig.~\ref{fig_ex3}. Here the magnitude of the probability density is shown by the color: darker regions correspond to larger probabilities. For a better visual effect, the color intensity was taken proportional to the logarithm of probability density. The presented results demonstrate the spontaneously stochastic behavior, because the probability density is almost identical for two very small values of the regularization parameter: $\nu = 10^{-5}$ (first row) and $\nu = 10^{-7}$ (second row). This provides a convincing numerical evidence that the inviscid limit exists and it is spontaneously stochastic.
The probability distributions have similar form at different times up to a proper scaling, in agreement with the self-similar limit (\ref{eqMainTh1_intro}) from Theorem~\ref{th1_conv}; see also Fig.~\ref{fig0}(b). The supplementary video~\cite{suppvideo} shows the evolution of probability density in time.

\section{Robust spontaneous stochasticity }\label{sec_results}

The major difficulty in applications of Theorem~\ref{th1_conv} to specific systems is how to verify the assumption of convergence to equilibrium (\ref{SRBEqDef1CE_s}), which is formulated for the attractor $\mathcal{A}'_+$ from Proposition~\ref{Th1_attr}. In this section, we discuss how specific and robust examples of systems satisfying this assumption can be constructed. 

Recall that system \eqref{yEqn}  must have a fixed point attractor $\mathcal{A}_-$. Let us choose a closed subset $\mathcal{V} \subset \mathbb{S}^{d-1}$, such that its complement $\mathbb{S}^{d-1}\setminus \mathcal{V}$ contains $\mathcal{A}_-$ and is contained in the interior of $\mathcal{B}(\mathcal{A}_-)$. The subset $\mathcal{V}$ contains basins of all the other attractors, in particular, $\mathcal{B}(\mathcal{A}_+) \subset \mathcal{V}$. It is convenient to use a diffeomorphism $h: \mathcal{V} \mapsto \hat{\mathcal{V}}$, which maps to a closed subset $\hat{\mathcal{V}} \subset \mathbb{R}^{d-1}$ and defines the new variable $\hat{\by} = h(\by)$. One can verify that systems \eqref{eq_RenX2} and \eqref{yEqn} keep the same form in terms of $\hat{\by}$, if we substitute $\bFF_\s$ and $F_\rr$ by the conjugated vector field $\hat{\bFF}_\s:\hat{\mathcal{V}} \mapsto \mathbb{R}^{d-1}$ and $\hat{F_\rr} = F_\rr \circ h^{-1}: \hat{\mathcal{V}} \mapsto \mathbb{R}$.
For simplicity, we will omit the hats in the notation below, therefore, assuming in all the relations that $\by \in \mathcal{V} \subset \mathbb{R}^{d-1}$. Although $\mathcal{V}$ is not forward invariant by the flow, this will not be necessary in what follows.

Consider now the attractor $\mathcal{A}_+$ of system \eqref{yEqn} with the physical measure $\mu_{{\mathsf{phys}}}$. Let us assume that it satisfies the \textit{convergence to equilibrium} property \eqref{SRBEqDef1CE_s}.


\begin{definition} 
\label{defrobust}
We say that the convergence to equilibrium property is $C^k$-robust, if there exists $\varepsilon > 0$ and a closed neighborhood $\mathcal{U}$ of the attractor, $\mathcal{A}_+ \subset \mathcal{U} \subset \mathcal{B}(\mathcal{A}_+)$, such that the following holds: for any $\varepsilon$-perturbation of $\bFF_\s$ 
in the $C^k$-topology, the corresponding system \eqref{yEqn} has an attractor contained in $\mathcal{U}$ having a physical measure and the convergence to equilibrium property.
\end{definition}

This definition extends naturally from the angular dynamics \eqref{yEqn} to the full auxiliary system \eqref{eq_RenX2} by considering perturbations of both $\bFF_\s$ and $F_\rr$. 
The following proposition provides  criterion that can be used for satisfying condition (\ref{SRBEqDef1CE_s}) in specific examples.

\begin{proposition}  \label{Prop2}
Let us assume that the attractor $\mathcal{A}_+$ in system \eqref{yEqn} has convergence to equilibrium and there exists a  constant $F_0 > 0$ such that $F_\rr(\by) = F_0$ for any $\by \in \mathcal{V}$. 

(i) If, for any $\by \in \mathcal{V}$,
    \be\label{conditionF0}
	\|\nabla \bFF_\s\| < (1-\alpha)F_0,
	\ee 
where $\|\nabla \bFF_\s\|$ is the operator norm of the Jacobian matrix $\nabla \bFF_\s$ at the point $\by$, then the attractor $\mathcal{A}'_+$ in system \eqref{eq_RenX2} has convergence to equilibrium.

(ii) If the convergence to equilibrium of $\mathcal{A}_+$ is $C^k$-robust and, for any $\by \in \mathcal{V}$,
\be\label{conditionF0b}
	\|\nabla \bFF_\s\| < \frac{(1-\alpha)F_0}{k},
	\ee
then the attractor $\mathcal{A}'_+$ has $C^k$-robust convergence to equilibrium.
\end{proposition}



Notice that conditions (\ref{conditionF0}) and (\ref{conditionF0b}) of Proposition~\ref{Prop2} can always be satisfied by a proper choice of the function $F_\rr$. This
suggests a constructive way for designing specific systems \eqref{ODEa}--\eqref{ODEb} having 
spontaneous stochasticity. For system to have $C^k$-robust spontaneous stochasticity, one should also impose that the fixed-point attractor $\mathcal{A}_-$ is hyperbolic, i.e., it persists under small perturbations of the system. 

Since the crucial hypothesis in this construction is that the attractor $\mathcal{A}_+$ has ($C^k$-robust) convergence to equilibrium, let us discuss examples of attractors having this property. The classical results on the ergodic theory of hyperbolic flows show that a $C^2$-hyperbolic attractor satisfying the C-dense condition of Bowen-Ruelle (density of the stable manifold of some orbit) has $C^2$-robust convergence to equilibrium, see \cite{BR}[Theorem 5.3].  In the last decades, many statistical properties have been studied for the larger class of singular hyperbolic attractors, which includes the  hyperbolic and the Lorenz attractors; see for example \cite{AP} as a basic reference and \cite{APPV09, AM16, AM19, AT20} for more recent advances. Robust convergence to equilibrium was naturally conjectured for such attractors \cite[problem E.4]{BDV}. Although the general proof is not available yet, recently in \cite[see Corollary B and Section 4]{AT20} were given examples of singular hyperbolic attractors having robust convergence to equilibrium, which include perturbations of the Lorenz attractor. In particular, it was shown there exists an arbitrary small $C^2$-perturbation of the Lorenz attractor so that the resulting system has $C^2$-robust convergence to equilibrium with respect to $C^1$-observables. 

Having in mind the above discussion, as a consequence of Proposition~\ref{Prop2} we obtain
\begin{corollary}\label{rob_stoc}
There exist examples exhibiting $C^2$-robust spontaneous stochasticity.
\end{corollary}

\section{Proofs}
\label{sec_proofs}

The central idea of the proofs is to reduce post-blowup dynamics of the stochastically regularized equations to the evolution of system (\ref{eq_RenX2}) over a time interval, which tends to infinity in 
the inviscid limit $\nu\searrow 0$. In this way, the inviscid limit is linked to the attractor and physical measure of system (\ref{eq_RenX2}). 

For the analysis of equations (\ref{eq_RenX2}), we transform them to a unidirectionally coupled dynamical system, whose decoupled part is the scale-invariant equation \eqref{yEqn}. Let us introduce the new temporal variable
	\begin{equation}
  	\label{eq_proof1d}
  	s(\tau) = \int_0^\tau w(\tau_1) \rmd \tau_1.
  	\end{equation}
Then,  system (\ref{eq_RenX2}) reduces to the so-called master-slave configuration 
	\begin{align}
	\label{eqU3firstCopy}
	\frac{\rmd \by}{\rmd s} & =  \bFF_\s(\by),\\[3pt] 
    	\frac{\rmd w}{\rmd s} &= 1+(\alpha-1)F_\rr(\by)w,
	\label{eqU3secondCopy}
	\end{align}
where the functions $\by(s)$ and $w(s)$ are written in terms of the new temporal variable $s$.
Note that the right-hand side of (\ref{eqU3secondCopy}) is unity for $w = 0$, which prevents $w(s)$ from changing the sign. Hence, $s$ in (\ref{eq_proof1d}) is a monotonously increasing function of $\tau$. Since $\bFF_\s$ and $F_\rr$ are bounded functions, solutions of system (\ref{eqU3firstCopy}) and (\ref{eqU3secondCopy}) are defined globally in time $s$. 

Notice that the new temporal variable (\ref{eq_proof1d}) is solution-dependent. This is a minor problem for the analysis of physical measures, which are related to temporal averages (\ref{SRBEqDef1Copy}). However, this is a serious obstacle for the property of convergence to equilibrium, which is associated with the ensemble average (\ref{SRBEqDef1CE_s}) at a fixed time.

\subsection{Proof of Proposition~\ref{Th1_attr}}

By the assumptions, system (\ref{eqU3firstCopy}) has the attractor $\mathcal{A}_+$. Therefore, we need to understand the dynamics of the second equation (\ref{eqU3secondCopy}). We are going to prove that this equation has the property of generalized synchronization between $\by$ and $w$. Let $X^s: \mathbb{S}^{d-1} \mapsto \mathbb{S}^{d-1}$ denote the flow of system (\ref{eqU3firstCopy}) and the pair ($X^s $, $X_w^s$) with $X_w^s: \mathbb{S}^{d-1} \times \mathbb{R}^+ \mapsto \mathbb{R}^+$ denote the flow of the system (\ref{eqU3firstCopy})--(\ref{eqU3secondCopy}). The formal definition of the generalized synchronization is the  asymptotic large-time relation~\cite{kocarev1996generalized}
	\begin{equation}
	\label{eq_proof5}
	w = G(\by) 
	\end{equation}
for $\by \in \mathcal{A}_+$ and
	\begin{equation}
	\label{eq_proof6}
	G(\by) = \lim_{s \to +\infty}X_{w}^s(X^{-s}(\by),w_0),
	\end{equation}
where the limit must be independent of an arbitrarily fixed $w_0 > 0$. These expressions yield the limiting value of $w(s)$ as $s \to +\infty$ under the condition that $\by(s) = \by$ is fixed, i.e., for the initial point $\by_0 = X^{-s}(\by)$. The generalized synchronization implies that $w(s)$ gets synchronized with the evolution of $\by(s)$. Our next step is to prove that the generalized synchronization occurs in our system with the continuous function $G(\by)$ given by (\ref{GdefCopy}).

The function $F_\rr:\mathbb{S}^{d-1}\to \mathbb{R}$ is continuous and therefore has an upper bound, $F_\rr(\by) < F_M$. Recall that the attractor $\mathcal{A}_+$ is a compact set with the defocusing property, $F_\rr(\by) > 0$ for any $\by \in \mathcal{A}_+$. Hence, we can choose a trapping neighborhood $\mathcal{U}_+$ of $\mathcal{A}_+$ and a positive constant $F_m$ such that 
	\begin{equation}
	\label{eq_proof_prop7}
	0 < F_m < F_\rr(\by) < F_M \quad  \textrm{for} \quad \by \in \mathcal{U}_+. 
	\end{equation}
We define the two quantities
	\begin{equation}
	\label{eq_WsmM}
	w_m = \frac{1}{(1-\alpha)F_M} > 0, \quad
	w_M = \frac{1}{(1-\alpha)F_m} > w_m.
	\end{equation}
For any $\by \in \mathcal{U}_+$, the derivative in (\ref{eqU3secondCopy}) satisfies the inequalities $dw/ds > 0$ for $0 < w \le w_m$ and $dw/ds < 0$ for $w \ge w_M$. Thus, the region 
	\begin{equation}
	\label{eq_proof_prop8}
	\mathcal{U}'_+ = \{ (\by,w): \ \by \in \mathcal{U}_+, \ w \in (w_m,w_M) \}
	\end{equation}
is trapping for system (\ref{eqU3firstCopy})--(\ref{eqU3secondCopy}), and it attracts any solution starting in $\mathcal{B}(\mathcal{A}_+) \times\mathbb{R}^+$. 

\begin{lemma}
\label{lemma1a}
The function 	\begin{equation}
	\label{GdefCopy_for_proof}
	G(\by) = 
	\int_0^{+\infty} \exp\left[(\alpha-1)\int_0^{s_1} 
	F_\rr\left(X^{-s_2}(\by)\right)
	\rmd s_2 \right]\rmd s_1
	\end{equation}
is continuous on the attractor $\mathcal{A}_+$.
\end{lemma}

\begin{proof}
Convergence of the integral in (\ref{GdefCopy_for_proof}) follows from the existence of positive lower bound $F_m$ in (\ref{eq_proof_prop7}) and the condition $\alpha < 1$. Now, let $\varepsilon > 0$. We are going to show that there exists $\delta > 0$, such that 
	\begin{equation}
	\label{eq_proof_lemma1}
	|G(\by')-G(\by)| < \varepsilon
	\end{equation}
for any $\by$ and $\by' \in \mathcal{A}_+$ with $|\by'-\by| < \delta$.
We split the integral in (\ref{GdefCopy_for_proof}) into two segments for $s_1 \in [0,s_p]$ and $s_1 \in [s_p,+\infty)$ with an arbitrary parameter $s_p > 0$. This yields 
\begin{equation}
\label{eq_P1proofBnC}
G(\by) = G_{s_p}(\by)+R_{s_p}(\by),
\end{equation}
where 
    \begin{equation}
    \label{eq_P1proofBn}
    G_{s_p}(\by) = \int_0^{s_p} \exp\left[(\alpha-1)
    \int^{s_1}_0 
    F_\rr\left(X^{-s_2}(\by)\right)
    \rmd s_2\right]\rmd s_1,
    \end{equation}
    \begin{equation}
    \label{eq_P1proofCadd}
    R_{s_p}(\by)
    = \int_{s_p}^{+\infty} 
    \exp\left[(\alpha-1)\int^{s_1}_0 
    F_\rr\left(X^{-s_2}(\by)\right)
    \rmd s_2\right]\rmd s_1.
    \end{equation}
Positive function (\ref{eq_P1proofCadd}) can be bounded using the property $F_\rr(\by) > F_m > 0$ from (\ref{eq_proof_prop7}) as
	\begin{equation}
	\label{eq_P1proofD}
	R_{s_p}(\by) < \int^{+\infty}_{s_p} 
	\exp\left[(\alpha-1)F_m s_1\right]\rmd s_1 
	= \frac{\exp\left[(\alpha-1)F_m s_p\right]}{(1-\alpha)F_m}.
	\end{equation}
By choosing 
    \begin{equation}
    \label{eq_P1proofE2}
    s_p > \frac{\log\left[(1-\alpha)
    F_m\varepsilon/4\right]}{(\alpha-1)F_m},
    \end{equation}
we have
    \begin{equation}
    \label{eq_P1proofE}
    R_{s_p}(\by) < \frac{\varepsilon}{4}.
    \end{equation}
This bound is valid for any $\by \in \mathcal{A}_+$. Using it in (\ref{eq_proof_lemma1}) with expression (\ref{eq_P1proofBnC}), we have
    \begin{equation}
    \label{eq_P1proofC}
    \big| G(\by')-G(\by) \big|
    < \big| G_{s_p}(\by')-G_{s_p}(\by)\big| +\frac{\varepsilon}{2}.
    \end{equation}
The function $G_{s_p}(\by)$ in (\ref{eq_P1proofBn}) contains integrations within finite intervals and, therefore, it is a continuous function defined for any $\by \in \mathbb{S}^{d-1} $. One can choose $\delta > 0$ such that $\big| G_{s_p}(\by')-G_{s_p}(\by)\big| < \varepsilon/2$ for any $\by$ and $\by' \in \mathbb{S}^{d-1} $ with $|\by'-\by| < \delta$. This yields the desired property (\ref{eq_proof_lemma1}) as the consequence of (\ref{eq_P1proofC}). 
\end{proof}

\begin{lemma}
\label{lemma2}
Function (\ref{eq_proof6}) takes the form (\ref{GdefCopy_for_proof})
for any $\by \in \mathcal{A}_+$ and $w_0 > 0$. Convergence of the limit (\ref{eq_proof6}) is uniform in the region
	\begin{equation}
	\label{eq_proof_prop9}
	\by \in \mathcal{A}_+, \quad w_0 \in (w_m,w_M).
	\end{equation}
For any solution $\by(s)$ of equation (\ref{eqU3firstCopy}) belonging to the attractor $\mathcal{A}_+$, the function $w(s) = G(\by(s))$ solves equation (\ref{eqU3secondCopy}).
\end{lemma}

\begin{proof}
Let us verify that equation (\ref{eqU3secondCopy}) has the explicit solution in the form
	\begin{equation}
	\label{eq_P1proofA}
	\begin{array}{rcl}
	w(s)  =  X_{w}^s(\by_0,w_0) 
	&=& 
	\displaystyle
	w_0\exp\left[(\alpha-1)\int^{s}_0 
	F_\rr\left(X^{s-s_2}(\by_0)\right)
	\rmd s_2\right]
	\\[12pt]
	&& 
	\displaystyle
	+\int^{s}_0 \exp\left[(\alpha-1)\int^{s_1}_0 F_\rr
	\left(X^{s-s_2}(\by_0)\right)
	\rmd s_2\right]\rmd s_1.
	\end{array}
	\end{equation}
It is easy to see that $w(0) = w_0$. The change of integration variable $\tilde{s}_2 = s_2-s$ yields
	\begin{equation}
	\label{eq_P1proofA_extra}
	\frac{\rmd }{\rmd s}
	\int^{s}_0 
	F_\rr\left(X^{s-s_2}(\by_0)\right) \rmd s_2 
	= F_\rr\left(X^{s}(\by_0)\right),\quad
	\end{equation}
	\begin{equation}
	\label{eq_P1proofA_extra2}
	\frac{\rmd }{\rmd s}
	\int^{s_1}_0 
	F_\rr\left(X^{s-s_2}(\by_0)\right) \rmd s_2 
	= F_\rr\left(X^{s}(\by_0)\right)-F_\rr\left(X^{s-s_1}(\by_0)\right).
	\end{equation}
Taking the derivative of (\ref{eq_P1proofA}) and using (\ref{eq_P1proofA_extra})--(\ref{eq_P1proofA_extra2}), we have
	\begin{equation}
	\label{eq_P1proofAder}
	\begin{array}{rcl}
	\displaystyle
	\frac{\rmd w}{\rmd s}
	&=& 
	\displaystyle
	w_0(\alpha-1)
	F_\rr\left(X^{s}(\by_0)\right)
	\exp\left[(\alpha-1)\int^{s}_0 
	F_\rr\left(X^{s-s_2}(\by_0)\right)
	\rmd s_2\right]
	\\[12pt]
	&& 
	\displaystyle
	+\,\exp\left[(\alpha-1)\int^{s}_0 F_\rr
	\left(X^{s-s_2}(\by_0)\right)
	\rmd s_2\right]
	\\[12pt]
	&& 
	\displaystyle
	+\,(\alpha-1)F_\rr\left(X^{s}(\by_0)\right)
	\int^{s}_0 
	\exp\left[(\alpha-1)\int^{s_1}_0 F_\rr
	\left(X^{s-s_2}(\by_0)\right)
	\rmd s_2\right]\rmd s_1
	\\[12pt]
	&& 
	\displaystyle
	-\,(\alpha-1)
	\int^{s}_0 
	F_\rr\left(X^{s-s_1}(\by_0)\right)
	\exp\left[(\alpha-1)\int^{s_1}_0 F_\rr
	\left(X^{s-s_2}(\by_0)\right)
	\rmd s_2\right]\rmd s_1.
	\end{array}
	\end{equation}
The term is the last line is integrated explicitly with respect to $s_1$ as
	\begin{equation}
	\label{eq_P1proofAder2}
	-\left.\exp\left[(\alpha-1)\int^{s_1}_0 F_\rr
	\left(X^{s-s_2}(\by_0)\right)
	\rmd s_2\right]\right|_{s_1 = 0}^{s_1 = s}= 
	1-\exp\left[(\alpha-1)\int^{s}_0 F_\rr
	\left(X^{s-s_2}(\by_0)\right)
	\rmd s_2\right].
	\end{equation}
Combining expressions (\ref{eq_P1proofA}), (\ref{eq_P1proofAder}) and (\ref{eq_P1proofAder2}) with $\by(s) = X^{s}(\by_0)$, one verifies that equation (\ref{eqU3secondCopy}) is indeed satisfied.

Note that $X^s(\by_0) \in \mathcal{A}_+$ for any $s \in \mathbb{R}$ and initial point on the attractor, $\by_0 \in \mathcal{A}_+$. Because of the positive lower bound $F_m$ in (\ref{eq_proof_prop7}) and $\alpha < 1$, the first term in the right-hand side of (\ref{eq_P1proofA}) vanishes in the limit $s \to +\infty$ uniformly for all initial points $\by_0 \in \mathcal{A}_+$ and $w_0 \in (w_m,w_M)$. For the same reason, the limit $s \to +\infty$ of the last term in (\ref{eq_P1proofA}) converges uniformly in this region. Therefore, taking the limit $s \to +\infty$ in (\ref{eq_P1proofA}) with $\by_0 = X^{-s}(\by)$ yields the equivalence of relations (\ref{GdefCopy_for_proof}) and (\ref{eq_proof6}).

Finally, consider solution (\ref{eq_P1proofA}) with $w_0 = G(\by_0)$ given by (\ref{GdefCopy_for_proof}). This yields 
	\begin{equation}
	\label{eq_P1proofF}
	\begin{array}{rcl}
	w(s)  
	&=& 
	\displaystyle
	\int_0^{+\infty} \exp\left[(\alpha-1)\int_{-s}^{s_1} 
	F_\rr\left(X^{-s_2}(\by_0)\right)
	\rmd s_2\right]\rmd s_1
	\\[12pt]
	&& 
	\displaystyle
	+\int^{s}_0 \exp\left[(\alpha-1)\int^{s_1}_0 F_\rr
	\left(X^{s-s_2}(\by_0)\right)
	\rmd s_2\right]\rmd s_1,
	\end{array}
	\end{equation}
where we combined the product of two exponents in the first term into the single one. After changing the integration variables $s_1 = s'_1-s$ and $s_2 = s'_2-s$ in the first integral term of (\ref{eq_P1proofF}), the full expression reduces to the simple form
	\begin{equation}
	\label{eq_P1proofFcomb}
	w(s) =
	\int_0^{+\infty} \exp\left[(\alpha-1)\int_{0}^{s_1} 
	F_\rr\left(X^{s-s_2}(\by_0)\right)
	\rmd s_2\right]\rmd s_1 = G(\by(s)),
	\end{equation}
where $G(\by)$ is given by formula (\ref{GdefCopy_for_proof}) and $\by(s) = X^s(\by_0)$.
Combining with the properties of the solution $w(s)$ described earlier, this proves the lemma.
\end{proof}

Lemma~\ref{lemma2} shows that $\mathcal{A}'_+$ from (\ref{eqAprime}) is the invariant set for system (\ref{eqU3firstCopy})--(\ref{eqU3secondCopy}). This set has the same structure of orbits as the attractor $\mathcal{A}_+$ of system (\ref{eqU3firstCopy}). We need to show that $\mathcal{A}'_+$ is an attractor with the trapping neighborhood (\ref{eq_proof_prop8}). Since $\mathcal{A}_+$ is the attractor of the first equation (\ref{eqU3firstCopy}), it is sufficient to prove that
	\begin{equation}
	\lim_{s \to +\infty} \big|w(s)-G(\tby(s))\big| = 0
	\label{eq_P1proofX1}
	\end{equation}
uniformly for all initial conditions $\by_0 \in \mathcal{A}_+$ and $w_0 \in (w_m,w_M)$. Since $\tby(s) = X^s(\by_0)$ and $w(s) = X_{w}^s(\by_0,w_0)$, we rewrite (\ref{eq_P1proofX1}) as
	\begin{equation}
	\lim_{s \to +\infty} \left|X_{w}^s(X^{-s}\big(\tby(s)\big),w_0)-G(\tby(s))\right| = 0.
	\label{eq_P1proofX1f}
	\end{equation}
The uniform convergence in this expression follows from Lemma~\ref{lemma2}.

It remains to prove the relations
	\begin{equation} 
	\label{eq_Th1_2_for_proof}
	\rmd  \mu'_{\mathsf{phys}}(\by,w) = 
	\frac{\delta\big(w-G(\by)\big)}{c\, G(\by)} \,\rmd \mu_{\mathsf{phys}}(\by)\,\rmd w, \quad
	c = \int \frac{\rmd \mu_{\mathsf{phys}}(\by)}{G(\by)}.
	\end{equation}
Because of the synchronization condition (\ref{eq_proof5}), the physical measure $\mu_{\syn}$ for the attractor $\mathcal{A}'_+$ of system (\ref{eqU3firstCopy})--(\ref{eqU3secondCopy}) is obtained from the physical measure $\mu_{\mathsf{phys}}$ of attractor $\mathcal{A}_+$ as
	\begin{equation}
	\rmd \mu_{\syn}(\by,w) = \delta\big(w-G(\by)\big)\,\rmd \mu_{\mathsf{phys}}(\by)\,\rmd w.
	\label{eq_P1proof9}
	\end{equation}
This measure corresponds to the dynamics of system  (\ref{eqU3firstCopy})--(\ref{eqU3secondCopy}). The time change $\rmd s = G(\by)\rmd \tau$ following from (\ref{eq_proof1d}) with $w = G(\by)$ transforms (\ref{eq_P1proof9}) to the physical measure (\ref{eq_Th1_2_for_proof}) for system (\ref{eq_RenX2}); see \cite[Ch.~10]{cornfeld2012ergodic}.

\subsection{Proof of Theorem~\ref{th1_conv}}

Let us consider the variables 
	\begin{equation}
  	\label{eq_A_w_Shifted}
  	w = (t-t^\nu)|\bx|^{\alpha-1},\quad
  	\tau = \log(t-t^\nu),
  	\end{equation}
where the temporal shift $t^\nu$, specified later in expression (\ref{eq_prTh1_B5}), depends on the regularization parameter $\nu > 0$.
Observe that $t^\nu$ was not present in the original definition (\ref{eqMainTh1_introB}), but it does not affect system (\ref{eq_RenX2}): at times $t > t^\nu$, each non-vanishing solution $\bx(t)$ of \eqref{ODEa}--\eqref{ODEb} is uniquely related to the solution $\by(\tau),~w(\tau)$ of system (\ref{eq_RenX2}) through the relations
    \begin{equation}
    \label{eq_Rmap_Shifted}
  	\bx = R_{t-t^\nu}(\by,w),\quad
	t = t^\nu+e^\tau.
    \end{equation}

Consider arbitrary times $t_2 > t_1 > t^\nu$ and denote
  	\begin{equation}\label{eq_prTh1_B2}
	\begin{array}{c}
	\bx_i = \bx(t_i), \quad
	\by_i = \by(t_i), \quad 
	w_i = w(t_i), \quad 
	\tau_i = \log(t_i-t^\nu),\quad i = 1,2.
	\end{array}
   	\end{equation}
Recalling that $\mathrm{\Phi}^t$ and $Y^\tau$ denote the flows of systems \eqref{ODEa}--\eqref{ODEb} and (\ref{eq_RenX2}), one has
  	\begin{equation}\label{eq_prTh1_B0}
	\bx_2 = \Phi^{t_2-t_1}\left(\bx_1\right),
	\quad t_2 > t_1 > t^\nu,
   	\end{equation}
and
  	\begin{equation}\label{eq_prTh1_B0b}
	(\by_2,w_2) = Y^{\tau_2-\tau_1}(\by_1,w_1),
	\quad \tau_2 > \tau_1.
   	\end{equation}
The first expression in (\ref{eq_Rmap_Shifted}) yields
  	\begin{equation}\label{eq_prTh1_B2prev}
	\bx_1 = R_{t_1-t^\nu}(\by_1,w_1), \quad
	\bx_2 = R_{t_2-t^\nu}(\by_2,w_2).
   	\end{equation}
Equalities (\ref{eq_prTh1_B0})--(\ref{eq_prTh1_B2prev}) provide the conjugation relation between the flows as
  	\begin{equation}\label{eq_prTh1_B3}
	\mathrm{\Phi}^{t_2-t_1} 
	= R_{t_2-t^\nu} \circ Y^{\tau_2-\tau_1} \circ R^{-1}_{t_1-t^\nu}.
   	\end{equation}
where $(\by,w) = R^{-1}_t(\bx)$ is the inverse map.

Let us apply relations (\ref{eq_prTh1_B2}) and (\ref{eq_prTh1_B3}) for the stochastically regularized solution given by (\ref{eqSR1_add2}) and (\ref{eqRR_5st}). We take 
  	\begin{equation}\label{eq_prTh1_B5}
	t_2 = t, \quad
	t_1 = t_\esc^\nu, \quad
	t^\nu = t_\esc^\nu-\nu^{1-\alpha} = t_\ent^\nu+(T-1)\nu^{1-\alpha}
   	\end{equation}
for any given time $t > t_b$. Notice that $t_2 > t_1$ for sufficiently small $\nu > 0$. Then, we use (\ref{eq_prTh1_B3}) to rewrite expression (\ref{eqRR_5st}) in the form of three successive measure pushforwards as
	\begin{equation}
	\label{eq_prTh1_B6}
	\mu^\nu_t(\bx) = 
	\left(\mathrm{\Phi}^{t_2-t_1}\right)_* 
	\mu^\nu_\esc(\bx)
	= \left(R_{t_2-t^\nu}\right)_* \left(Y^{\tau_2-\tau_1}\right)_* 
	\left(R^{-1}_{t_1-t^\nu}\right)_*
	\mu^\nu_\esc(\bx).
	\end{equation}
For the first pushforward, expressions (\ref{eq_prTh1_B5}) yield
	\begin{equation}
	\label{eq_prTh1_B6y}
	\left(R^{-1}_{t_1-t^\nu}\right)_* \mu^\nu_\esc(\bx)
	= \left(R^{-1}_{\nu^{1-\alpha}}\right)_* \mu^\nu_\esc(\bx). 
	\end{equation}
Notice from \eqref{eqMainTh1_introB} that $R^{-1}_{\nu^{1-\alpha}}(\bx) = R^{-1}_1(\bx/\nu)$. Thus, applying expressions (\ref{eqSR1_add2}), we reduce (\ref{eq_prTh1_B6y}) to the form
	\begin{equation}
	\label{eq_prTh1_B6z}
	\left(R^{-1}_{t_1-t^\nu}\right)_* \mu^\nu_\esc(\bx)
	= \left(R^{-1}_1\right)_* \mu^\nu_f(\bx), \quad
	\rmd \mu_f^\nu(\bx) = f^\nu_\esc(\bx)\rmd \bx,
	\end{equation}
where $\mu^\nu_f$ denotes the absolutely continuous probability measure with the density $f^\nu_\esc$. Finally, using expressions (\ref{eq_prTh1_B2}), (\ref{eq_prTh1_B5}) and (\ref{eq_prTh1_B6z}) in (\ref{eq_prTh1_B6}), yields
	\begin{equation}
	\label{eq_prTh1_B6x}
	\mu^\nu_t(\bx) 
	= 
	\left(R_{t-t^\nu}\right)_* 
	\left(Y^{\tau^\nu}\right)_* \left(R^{-1}_1\right)_*
	\mu^\nu_f(\bx)
	\end{equation}
with 
  	\begin{equation}\label{eq_prTh1_B7}
	\tau^\nu = \tau_2-\tau_1 = \log \frac{t-t_\ent^\nu-(T-1)\nu^{1-\alpha}}{\nu^{1-\alpha}}.
   	\end{equation}
In the inviscid limit, from relations (\ref{eq_invlim}), (\ref{eq_prTh1_B5}) and (\ref{eq_prTh1_B7}) one has
  	\begin{equation}\label{eq_prTh1_B7x}
	\lim_{\nu \searrow 0} t^\nu = t_b,\quad
	\lim_{\nu \searrow 0} \tau^\nu = +\infty.
   	\end{equation}

It remains to take the limit $\nu \searrow 0$ in (\ref{eq_prTh1_B6x}). The convergence of entry times from (\ref{eq_invlim}) and Proposition~\ref{propSol} yield
	\begin{equation}
	\label{eq_prTh1_10}
	\lim_{\nu \searrow 0} \by_\ent^\nu = \by_-,
	\end{equation}
where $\mathcal{A}_- = \{\by_-\}$ denotes the fixed-point attractor and $\by_\ent^\nu = \bx_\ent^\nu/\nu$ correspond to entry points.
Since the map $\mathrm{\Psi}_\rnd$ in (\ref{eqSR1}) is continuous, the limit (\ref{eq_prTh1_10}) implies
  	\begin{equation}\label{eq_prTh1_6f}
	f_\esc^\nu \xrightarrow{L^1} f_-
	\quad \textrm{as} \quad \nu \searrow 0,
  	\end{equation}
where
	\begin{equation}
	\label{eq_prTh1_6g}
	f_\esc^\nu = \mathrm{\Psi}_\rnd (\by_\ent^\nu),\quad
	f_- = \mathrm{\Psi}_\rnd (\by_-).
	\end{equation}
Using this limiting function, we rewrite (\ref{eq_prTh1_B6x}) as
	\begin{equation}
	\label{eq_prTh1_B11}
	\mu^\nu_t(\bx) 
	=
	\left(R_{t-t^\nu}\right)_* \left[
	\left(Y^{\tau^\nu}\right)_* \left(R^{-1}_1\right)_* \mu_-(\bx)+
	\left(Y^{\tau^\nu}\right)_* \left(R^{-1}_1\right)_* \Delta\mu_f^\nu(\bx)
	\right],
	\end{equation}
where we introduced the probability measure $\rmd \mu_-(\bx) = f_-(\bx)\rmd \bx$ and the signed measure  for the difference $\Delta\mu_f^\nu(\bx) = \mu_f^\nu(\bx)-\mu_-(\bx)$. Now we can take the inviscid limit $\nu \searrow 0$ for the expression in square parentheses of equation (\ref{eq_prTh1_B11}), where the times of pushforwards behave as (\ref{eq_prTh1_B7x}). Since the measure $\left(R^{-1}_1\right)_* \mu_-(\bx)$ does not depend on $\nu$, the first term in square parentheses converges to $\mu'_{\mathsf{phys}}$ by the convergence to equilibrium property. The remaining term vanishes in the limit $\nu \searrow 0$, because the flow conserves the $L^1$ norm of the density function, and this norm vanishes by the property (\ref{eq_prTh1_6f}). This yields the limit (\ref{eqTh1_1}) with the measure (\ref{eqMainTh1_intro}). Properties (\ref{eqTh1_2}) and (\ref{eqTh1_3}) follow directly from the definitions (\ref{eqRR_5st}) and (\ref{eqMainTh1_introB}).

\subsection{Proof of Proposition \ref{Prop2}} 

We will formulate the proof for the first part of proposition, such that it can be extended later for $C^k$-perturbed systems.

System \eqref{eq_RenX2} considered for $(\by,w) \in \mathcal{V} \times \mathbb{R}^+$ with $F_\rr(\by) \equiv F_0$ takes the form
	\begin{equation}
	\label{eqU3Uadd1}
	\frac{\rmd \by}{\rmd \tau} =  w\bFF_\s(\by),\qquad
	\frac{\rmd w}{\rmd \tau} = w+(\alpha-1)F_{0}w^2.
	\end{equation}
The second equation in (\ref{eqU3Uadd1}) has the fixed point attractor $w = W_0 := \left[(1-\alpha)F_{0}\right]^{-1} > 0$ with the basin $w > 0$. 
Recall that $\mathcal{B}(\mathcal{A}_+) \subset \mathcal{V} \subset \mathbb{R}^{d-1}$. 
Expression (\ref{GdefCopy}) with $F_\rr(\by) \equiv F_0$ defines the function $G:\mathcal{B}(\mathcal{A}_+) \mapsto \mathbb{R}^+$ as
	\begin{equation}
	\label{eqP41_fgraph}
	 G(\by) \equiv W_0.
	\end{equation}
We define the corresponding graph as
	\begin{equation}
	\label{eqP41_graph}
	 \mathcal{G}(\mathcal{A}_+) 
	 = \{(\by,w): \by \in \mathcal{B}(\mathcal{A}_+),\, w = G(\by)\},
	\end{equation}
which is the invariant manifold for system (\ref{eqU3Uadd1}). The attractor $\mathcal{A}'_+ \subset  \mathcal{G}(\mathcal{A}_+)$ is given by (\ref{eqAprime}).

Linearization of system \eqref{eqU3Uadd1} at any point of $\mathcal{G}(\mathcal{A}_+)$ takes the form
\begin{equation}
\label{eqU3U_J2add2x}
\frac{d}{\rmd \tau} 
\left(\begin{array}{c}
\delta{\by}\\[3pt] \delta{w}
\end{array}\right)
= \left(\begin{array}{cc}
W_0\nabla\mathbf{F}_s
&\ \ \mathbf{F}_s(\by)
\\[5pt]
\mathbf{0}& -1
\end{array}\right)
\left(\begin{array}{c}
\delta\by\\[3pt] \delta{w}
\end{array}\right),
\end{equation}
where $(\delta\by,\delta w) \in \mathbb{R}^{d-1} \times \mathbb{R}$ is an infinitesimal perturbation in the tangent space. It is straightforward to verify that system (\ref{eqU3U_J2add2x}) has a solution 
\begin{equation}
\label{eqU3U_J2add2c}
\left(\begin{array}{c}
\delta\by\\[3pt] \delta{w}
\end{array}\right)
= e^{-\tau}\left(\begin{array}{c}
\mathbf{F}_s(\by) \\[3pt] -1
\end{array}\right),
\end{equation}
which provides the eigenvalue  $-1$ with the corresponding eigenvector. The eigenvector defines the linear space $E^{ss}$ transversal to the graph $\mathcal{G}(\mathcal{A}_+)$, and it will play the role of strong stable (contracting) direction. 
Remaining eigenvalues are determined by the Jacobian matrix $W_0\nabla\mathbf{F}_s$ with the corresponding linear invariant space $E^c = \mathbb{R}^{d-1} \times \{0\}$ tangent to the graph $\mathcal{G}(\mathcal{A}_+)$.
Assumptions (\ref{conditionF0}) imply that eigenvalues of $W_0\nabla\mathbf{F}_s$ with  $W_0 = \left[(1-\alpha)F_{0}\right]^{-1}$ have absolute values smaller than unity. 

We showed that, at each point of the graph (\ref{eqP41_graph}), there exists a splitting $E^{ss}\oplus E^c$ of the tangent space, which is invariant for linearized system (\ref{eqU3U_J2add2x}) 
and such that $E^{ss}$ dominates (contracts stronger than) the so-called central directions in $E^c$. 
It follows from the stable manifold theorem that each point of $\mathcal{G}(\mathcal{A}_+)$ has a one-dimensional strong stable invariant manifold, which is tangent to $E^{ss}$; for background on the invariant manifold theory see \cite[Chapter 6]{S87}  for
discrete systems and \cite[Section 4.5]{Wigg94} for flows. 
Such structure can be described locally by a homeomorphism $\rho: \mathcal{U}\times (-\delta, \delta) \mapsto \mathcal{U}'$, where $\mathcal{U}$ and $\mathcal{U}'$ are, respectively, some trapping neighborhoods of the attractors $\mathcal{A}_+$ and $\mathcal{A}'_+$, and $\delta > 0$ is some (small) number. Here, the fibers $\rho(\by,\xi)$ for fixed $\by$ are local $C^1$-parametrizations of the strong stable manifolds starting on the graph $\rho(\by,0) \in \mathcal{G}(\mathcal{A}_+)$. 

Let $Y^\tau$ by the flow of system (\ref{eqU3Uadd1}). We denote by $Y^\tau_{\rho}=\rho^{-1}\circ Y^\tau  \circ \rho$ the flow, which is defined in $\mathcal{U}\times (-\delta, \delta)$ and conjugated to $Y^\tau$. By construction, this new flow $Y^\tau_{\rho}$ has the attractor $\mathcal{A}_\rho = \{(\by,0):\ \by \in \mathcal{A}_+\}$ with the physical measure     		
	\begin{equation}
    	\label{eq_proofR0z}
	\rmd \mu_\rho(\by,\xi) = \rmd \mu_{\mathsf{phys}}(\by)\, \delta(\xi)\rmd \xi, 
    	\end{equation}
where $\delta(\xi)$ is the Dirac delta-function and $\mu_{\mathsf{phys}}$ is the physical measure of the attractor $\mathcal{A}_+$. Straight segments $(\by,\xi)$ with fixed $\by$ and $\xi \in (-\delta,\delta)$ correspond to strong stable manifolds for the new flow $Y^\tau_{\rho}$. Moreover, since strong stable manifolds have constant eigenvalue $-1$, $Y^\tau_{\rho}$ has uniform contraction along strong stable manifolds to the plane $\xi = 0$ in a sufficiently small neighborhood $\mathcal{U}\times (-\delta, \delta)$.

Now, the property of convergence to equilibrium for the flow $Y^\tau$ follows from the same property for $Y^\tau_{\rho}$, where the latter is established as follows. The condition of convergence to equilibrium (\ref{SRBEqDef1CE_s}) for the new system becomes
    	\begin{equation}
    	\lim_{\tau\rightarrow +\infty}\int \varphi \circ Y^\tau_\rho
	\, \rmd \mu(\by,\xi)
    	= \int \varphi \,\rmd \mu_\rho(\by,\xi)
    	= \int \varphi(\by,0) \,\rmd \mu_{\mathsf{phys}}(\by),
    	\label{SRBEqDef1CENz}
    	\end{equation}
where we used (\ref{eq_proofR0z}) and integrated the Dirac delta-function.
It is enough to verify this condition for absolutely continuous probability measures $\mu(\by,\xi)$ supported in the trapping neighborhood $\mathcal{U}\times (-\delta, \delta)$.
Using properties of strong stable manifolds for the flow $Y^\tau_\rho$, the integral in the left-hand side of (\ref{SRBEqDef1CENz}) can be written as
    	\begin{equation}
    	\int \varphi \circ Y^\tau_\rho
	\, \rmd \mu(\by,\xi)
    	=
    	\int {\varphi \big( Y^\tau_\rho(\by,0) \big) 
	\, \rmd \mu(\by,\xi)}
    	+\int \varphi_1 \circ  Y^\tau_\rho
	\, \rmd \mu(\by,\xi),
    	\label{eq_proofR1z}
    	\end{equation}
where we introduced the function $\varphi_1(\by,\xi) = \varphi(\by,\xi)-\varphi(\by,0)$.
Since the flow $Y^\tau_\rho$ has the property of uniform contraction to the plane $\xi = 0$, where $\varphi_1 = 0$, the last integral in (\ref{eq_proofR1z}) vanishes in the limit $\tau \to +\infty$. For the first integral in the right-hand side of (\ref{eq_proofR1z}), we write
    	\begin{equation}
    	\int {\varphi \big( Y^\tau_\rho(\by,0) \big) 
	\, \rmd \mu(\by,\xi)}
    	= \int {\varphi \big( Y^\tau_\rho(\by,0) \big) 
	\, \rmd \mu_{\mathsf{int}}(\by)},
    	\label{eq_proofR2z}
    	\end{equation}
where $\mu_{\mathsf{int}}(\by)$ is obtained from the measure $\mu(\by,\xi)$ by integration with respect to $\xi$. The last integral in (\ref{eq_proofR2z}) corresponds to the flow $Y^\tau_\rho$ restricted to the invariant plane $\xi = 0$, and it is conjugate to the original flow $Y^\tau$ restricted to the graph (\ref{eqP41_graph}) with the constant function (\ref{eqP41_fgraph}). The latter becomes the flow $X^s$ of system \eqref{yEqn} after the scaling of time with the constant factor $W_0$. Therefore, we reduced (\ref{SRBEqDef1CENz}) to the analogous condition of convergence to equilibrium for system \eqref{yEqn}, which holds by our assumptions. This proves the first part of the proposition.

For the proof of the $C^k$-robust convergence to equilibrium, we will need the following 

\begin{lemma}
\label{lemma1}
Consider an attractor $\mathcal{A}_+$ of system \eqref{yEqn} with $C^k$-functions $\bFF_\s: \mathcal{V}\mapsto \mathbb{R}^{d-1} $ and $F_\rr: \mathcal{V} \mapsto \mathbb{R}$ satisfying the conditions
	\be\label{conditionFm} 
		\| \nabla \bFF_\s\| < M,\quad
		F_\rr(\by) > m > 0
	\ee 
for any $\by \in \mathcal{B}(\mathcal{A}_+)$ and positive constants $m$ and $M$ such that $M < (1-\alpha)m/k$. Then,
\begin{enumerate}
\item[(i)] Expression \eqref{GdefCopy} defines the $C^k$-differentiable function $G:\mathcal{B}(\mathcal{A}_+) \mapsto \mathbb{R}^+$.
\item[(ii)] Let $\by(\tau)$ be the solution of equation
\begin{equation}\label{eqU3firstbar}
\frac{\rmd \by}{\rmd \tau}= G(\by)\bFF_\s(\by)
\end{equation}
for arbitrary initial condition $\by_0 \in \mathcal{B}(\mathcal{A}_+)$. Then $w(\tau) = G(\by(\tau))$ satisfies \eqref{eq_RenX2}.
\item[(iii)] 
Sufficiently small $C^k$-perturbations of $\bFF_\s$ and $F_\rr$ yield small $C^k$-perturbations of $G$.
\end{enumerate}
\end{lemma}

\begin{proof}
The above lemma is related to the general statements of the invariant manifold theory as stated in \cite{HPS} for discrete systems and in \cite{Wigg94} for flows. Below, for completeness, we present a direct proof for arbitrary functions $F_\rr$ satisfying (\ref{conditionFm}).
Let us first consider the case $k = 1$.

($i$) Changing signs of the integration variables $s_1$ and $s_2$ in expression (\ref{GdefCopy}) yields
	\begin{equation}
	\label{Gdef2}
	G(\by) = \lim_{s \to -\infty} G_s(\by),\qquad
	G_s(\by) = \int_s^0 \exp\left[(\alpha-1)\int_{s_1}^0 
	F_\rr\circ X^{s_2}(\by)
	\rmd s_2\right]\rmd s_1,
	\end{equation}
where we introduced the function $G_s: \mathcal{V} \mapsto \mathbb{R}^+$. By construction, $G_s$ is a $C^1$-function for any $s$. The second condition in (\ref{conditionFm}) implies the uniform convergence of the limit (\ref{Gdef2}) for $\by \in \mathcal{B}(\mathcal{A}_+)$. Hence, the limiting function $G$ is continuous in $\mathcal{B}(\mathcal{A}_+)$.

Computing the Jacobian matrix $\nabla G_s$ in (\ref{Gdef2}) at a given point $\by$ yields
	\begin{equation} \label{Ggrad}
	\nabla G_s =  (\alpha-1) \int_{s}^0
	\left(\int_{s_1}^0 \nabla  \left( F_\rr\circ X^{s_2}\right) \rmd s_2\right) 
	\exp\left[(\alpha-1)\int_{s_1}^0 F_\rr\circ X^{s_2}(\by)\rmd s_2\right]\rmd s_1,
	\end{equation}
where
	\begin{equation} \label{Ggrad_2}
	\nabla  \left( F_\rr\circ X^{s_2}\right) 
	=  \left(\nabla F_\rr\right)_{X^{s_2}(\by)} \nabla X^{s_2},	
	\end{equation}
and $\left(\nabla F_\rr\right)_{X^{s_2}(\by)}$ denotes the gradient vector $\nabla F_\rr$ computed at $X^{s_2}(\by)$. 

Since $X^s$ is the flow of system \eqref{yEqn}, by the classical theory of ordinary differential equations, the Jacobian matrix $\nabla X^s$ satisfies the linear Cauchy problem
	\begin{equation} \label{Ggrad_3}
	\frac{\rmd }{\rmd s} \,\nabla X^s 
	= 
	\left(\nabla \bFF_\s\right)_{X^{s}(\by)} \nabla X^s, \qquad
	\nabla X^0 = \mathbf{I},
	\end{equation}
where $\mathbf{I}$ is the identity matrix and $\left(\nabla \bFF_\s\right)_{X^{s}(\by)}$ is the Jacobian matrix $\nabla\bFF_\s$ at $X^{s}(\by)$. Using (\ref{Ggrad_3}) for negative $s$, the first bound of (\ref{conditionFm}) and Gr\"onwall's inequality, we estimate
	\begin{equation} \label{Ggrad_4}
	\| \nabla X^{s} \| \leq e^{-M s},\qquad s \le 0.
	\end{equation}
Let $M_\rr = \max_{\by \in \mathcal{V}} \|\nabla F_\rr\| \ge 0$.
Using expressions (\ref{Ggrad}), (\ref{Ggrad_2}) and (\ref{Ggrad_4}) with the bounds (\ref{conditionFm}) and recalling that $\alpha < 1$, $s_1 \le 0$ and $s_2 \le 0$, we obtain 
	\begin{equation}
	\|\nabla G_s-\nabla G_{s'} \| \leq  
	(1-\alpha) \int_{s}^{s'}
	\left(\int_{s_1}^0  M_\rr e^{-M s_2} \rmd s_2\right) 
	e^{(1-\alpha)m s_1}\rmd s_1
	\label{Ggrad_5}
	\end{equation}
for any $s < s' < 0$.  Integrating the right-hand side of (\ref{Ggrad_5}) and taking into account that 
	\begin{equation} 
	\label{Ggrad_6_extra}
	(1-\alpha)m > (1-\alpha)m-M > 0 
	\end{equation}
by the conditions of the lemma, one can show the Cauchy convergence of the gradients $\nabla G_s$ in the limit $s \to -\infty$. Since the bound in (\ref{Ggrad_5}) does not depend on $\by$, the convergence is uniform in $\mathcal{B}(\mathcal{A}_+)$. This proves the continuity of the limiting gradient $\nabla G$ in (\ref{Gdef2}).

($ii$) Consider the pair of functions $\by(\tau)$ and $w(\tau) = G(\by(\tau))$, where $\by(\tau)$ satisfies equation (\ref{eqU3firstbar}). Obviously, these functions satisfy the first equation of (\ref{eq_RenX2}). The second equation in (\ref{eq_RenX2}) can be transformed to the form (\ref{eqU3secondCopy}) with the time change (\ref{eq_proof1d}). Then, this equation is verified as in Lemma~\ref{lemma2}, taking into account that the integrals converge uniformly for all $\by \in  \mathcal{B}(\mathcal{A}_+)$.
\vspace{2mm} 

($iii$) Using the uniform bound (\ref{Ggrad_5}) one proves that the convergence of integrals in (\ref{Ggrad}) as $s \to -\infty$ is uniform not only with respect to $\by$, but also with respect to sufficiently small $C^1$-perturbations of the functions $\bFF_\s(\by)$ and $F_\rr(\by)$.  This implies that such perturbations lead to  $C^1$-perturbations of $G(\by)$. 

This proof extends to the $C^k$ case for $k > 1$ by computing high-order derivatives of $G$ in the way similar to (\ref{Gdef2}) and (\ref{Ggrad}). Generalizing expression (\ref{Ggrad_4}), one can show that $k$th-order derivatives of $X^s(\by)$ are bounded by $c\exp(-kMs)$ for $s \le 0$ and some coefficient $c > 0$.  We leave details of this rather straightforward derivation to the interested reader.
\end{proof}

Consider now a perturbed system \eqref{eq_RenX2} with $\tilde{\bFF}_s$ close to $\bFF_\s$ and $\tilde{F}_r$ close to $F_\rr$
in the $C^k$-metric; here and below the tildes denote properties of the perturbed system. Conditions of Definition~\ref{defrobust} ensure that the perturbed system \eqref{yEqn} has an attractor $\tilde{\mathcal{A}}_+$ with the physical measure and the convergence to equilibrium property. In turn, the perturbed system (\ref{eq_RenX2}) has the attractor $\tilde{\mathcal{A}}'_+$ given by the graph $w = \tilde{G}(\by)$ of $\by \in \tilde{\mathcal{A}}_+$; see Proposition~\ref{Th1_attr}.
Conditions (\ref{conditionF0b}) remain valid if the perturbation is sufficiently small. Hence, one can choose $m$ and $M$ satisfying conditions of Lemma~\ref{lemma1}, establishing that the function $\tilde{G}(\by)$ is $C^k$-close to the constant from (\ref{eqP41_fgraph}), and also the graph $w = \tilde{G}(\by)$ with $\by \in \mathcal{B}(\tilde{\mathcal{A}}_+)$ is invariant under the flow of  perturbed system (\ref{eq_RenX2}). 

Restriction of (\ref{eq_RenX2}) to the invariant hyper-surface $w = \tilde{G}(\by)$ yields 
    \begin{equation}
    \label{eqPr4A}
    \frac{\rmd \by}{\rmd \tau}= \tilde{G}(\by)\tilde{\bFF}_s(\by).
    \end{equation}
This system is $C^k$-close to $\rmd \by/\rmd \tau= W_0 \bFF_\s(\by)$, where the latter is equivalent to $\rmd \by/\rmd s = \bFF_\s(\by)$ up to the constant time scaling. Since the attractor $\mathcal{A}_+$ of unperturbed system \eqref{yEqn} is assumed to have a physical measure with the $C^k$-robust convergence to equilibrium, the attractor $\tilde{\mathcal{A}}_+$ of perturbed system (\ref{eqPr4A}) has a physical measure with the property of convergence to equilibrium, provided that the perturbation is sufficiently small. For concluding the proof, one should notice that all arguments in the first part of the proof (based on the invariant manifold theory) remain valid for small $C^k$ perturbations of the system and of the graph (\ref{eqP41_graph}).

\vspace{3mm}\noindent\textbf{Acknowledgements:}
\ We would like to thank Isaia Nisoli and Enrique Pujals for helpful discussions.  The research of TDD was partially supported by the NSF
DMS-2106233 grant. AAM was supported by CNPq (grants 303047/2018-6, 406431/2018-3).

\bibliographystyle{siam} 
\bibliography{refs}

\end{document}